\documentclass[11pt]{amsart}
\usepackage{amssymb,amsmath,txfonts,mathrsfs,mathtools,titletoc, tikz, stmaryrd}
\usepackage{color}
\usepackage{graphicx}
\usepackage{float}
\usepackage{esint}
\usepackage{hyperref}

\usepackage[alphabetic]{amsrefs}
\newtheorem{theorem}{Theorem}[section]
\newtheorem{prop}[theorem]{Proposition}
\newtheorem{lemma}[theorem]{Lemma}
\newtheorem{remark}[theorem]{Remark}

\newtheorem{definition}[theorem]{Definition}
\newtheorem{notation}[theorem]{Notation}

\DeclareMathOperator{\ep}{\epsilon}

\numberwithin{equation}{section}

\def\pf{{\it Proof:}~}
\def\R{{\bf\mathbb R} }
\begin{document}

\title[The first Neumann eigenvalue and the width]{The first Neumann eigenvalue and the width}
\author{Haibin Wang, Guoyi Xu}
\address{Haibin Wang\\ Department of Mathematical Sciences\\Tsinghua University, Beijing\\P. R. China}
\email{wanghb20@mails.tsinghua.edu.cn}
\address{Guoyi Xu\\ Department of Mathematical Sciences\\Tsinghua University, Beijing\\P. R. China}
\email{guoyixu@tsinghua.edu.cn}
\date{\today}
\date{\today}

\begin{abstract}
We prove the sharp lower bound of the first Neumann eigenvalue for bounded convex planar domain in term of its diameter and width.
\\[3mm]
Mathematics Subject Classification: 34B09, 35J25, 35J05.
\end{abstract}
\thanks{G. Xu was partially supported by NSFC 12141103.}

\maketitle

\titlecontents{section}[0em]{}{\hspace{.5em}}{}{\titlerule*[1pc]{.}\contentspage}
\titlecontents{subsection}[1.5em]{}{\hspace{.5em}}{}{\titlerule*[1pc]{.}\contentspage}
\tableofcontents

\section{Introduction}

Payne and Weinberger \cite{PW} firstly showed that for any convex domain $\Omega\subseteq \mathbb{R}^n$ with fixed diameter $=2$,  its first Neumann eigenvalue $\mu_1(\Omega)$ is at least $\frac{\pi^2}{4}$.  Their method is by integral estimates and geometric cutting of the region, one key is the eigenvalue estimate for weighted ODE corresponding the region $\Omega$. 

Li \cite{Li} and Li-Yau \cite{LY} applied the gradient estimates of eigenfunctions to study the lower bound of the first Neumann (Dirichlet) eigenvalue on manifolds with non-negative Ricci curvature. Through a delicate refined gradient estimate of eigenfunctions, Zhong-Yang \cite{ZY} firstly obtain the sharp lower bound of the first eigenvalue for compact Riemannian manifolds with non-negative Ricci curvature.  Zhong-Yang's type gradient estimate of eigenfunction is closely related to the solution of suitable ODE (also see \cite{Kroger}). There is another proof the sharp lower bound of the first eigenvalue by two-points method in \cite{AC}. 

Hang-Wang \cite{HW} firstly showed that $\mu_1(\Omega)= \frac{\pi^2}{4}$ if and only if $\Omega$ is isometric to $[-1, 1]\subseteq \mathbb{R}$.  The sharp lower bound of the first Neumann eigenvalue and its rigidity later is obtained through optimal transport \cite{Klartag}.

\begin{definition}\label{def width of region in Rn}
{We define the width of $\Omega\subseteq \mathbb{R}^n$ as follows:
\begin{align}
\mathrm{Width}(\Omega)= \min_{v\subseteq \mathbb{S}^{n- 1}} \max_{t\in \mathbb{R}} \mathrm{diam}(P_{v, t}\cap \Omega) , \nonumber
\end{align}
where $P_{v, t}\vcentcolon= \{x\in \mathbb{R}^n: (x- tv)\cdot v= 0\}$ and $\displaystyle \mathrm{diam}(P_{v, t}\cap \Omega)\vcentcolon= \max_{x, y\in P_{v, t}\cap \Omega}|x- y|_{\mathbb{R}^n}$.

Furthermore, we define the projective width of $\Omega\subseteq \mathbb{R}^n$ as follows:
\begin{align}
&\mathscr{P}_{v}(\Omega)\vcentcolon= \{x\in \mathbb{R}^n: x\cdot v= 0, x= y+ s\cdot v \ \text{for some}\ y\in \Omega, s\in \mathbb{R}\}, \quad  \forall v\in \mathbb{S}^{n- 1}\subseteq \mathbb{R}^n; \nonumber \\
&\mathscr{PW}(\Omega)= \min_{v\subseteq \mathbb{S}^{n- 1}}  \mathrm{diam}(\mathscr{P}_{v}(\Omega)) , \nonumber
\end{align}
where $\displaystyle \mathrm{diam}(\mathscr{P}_{v}(\Omega)) \vcentcolon= \max_{x, y\in \mathscr{P}_{v}(\Omega)}|x- y|_{\mathbb{R}^n}$.
}
\end{definition}

The main result of this paper is the following Theorem.
\begin{theorem}\label{thm NW ineq}
{There are universal constants $C_1, C_2> 0$, such that for any convex domain $\Omega\subseteq \mathbb{R}^2$ with $\mathrm{diam}(\Omega)= 2$ and $\mathrm{Width}(\Omega)\in (0, C_1)$; we have 
\begin{align}
\mu_1(\Omega)\geq \frac{\pi^2}{4}+ C_2\cdot \mathrm{Width}(\Omega)^2.  \label{nw ineq-main} 
\end{align}
}
\end{theorem}

\begin{remark}\label{rem general width case}
If $\mathrm{Width}(\Omega)\geq C_1$, by compactness argument and the rigidity result from \cite{HW}, we can get (\ref{nw ineq-main}). Therefore, the above Theorem confirms a conjecture of Fengbo Hang and Xiaodong Wang \cite{HW} (also see \cite{Sakai}).

In fact (\ref{nw ineq-main}) can be obtained for convex domain $\Omega\subseteq \mathbb{R}^2$ with $\mathrm{diam}(\Omega)= 2$, by compactness argument and the rigidity in \cite{HW}. The main contribution of this paper is a quantitative argument for (\ref{nw ineq-main}) when the width is uniformly small with respect to the diameter of $\Omega$.

The order of width is sharp as $\mathrm{Width}(\Omega)\to 0$. We choose $\Omega=[0,\sqrt{4-\epsilon^2}]\times[0,\epsilon]$, then $\mu_1(\Omega)-\frac{\pi^2}{4}=\frac{\epsilon^2}{4(4-\epsilon^2)}$.
\end{remark}

\begin{remark}\label{rem general dim case}
{The results of this paper are part of the first author's Ph.D thesis. During the process of writing this paper, Amato, Bucur, Fragalà had posted their paper \cite{ABF} on arXiv. Among other things, the above theorem is proved in a more general context in \cite{ABF} by different method. 
}
\end{remark}

For the study of lower bound of the first Dirichlet eigenvalue, Hersch \cite{Hersch} showed $\displaystyle \lambda_1(\Omega)\cdot \mathrm{IR}(\Omega)^2- \frac{\pi^2}{4}\geq 0$, where $\Omega\subseteq \mathbb{R}^2$ is a convex region, $\lambda_1(\Omega)$ is the first Dirichlet eigenvalue of $\Omega$ and $\displaystyle \mathrm{IR}(\Omega)\vcentcolon = \sup_{x\in \Omega} d(x,  \partial \Omega)$ is the inscribed radius of $\Omega$. 

Later, for convex region $\Omega\subseteq \mathbb{R}^n$, Mendez-Hernandez \cite{MH}  (also see Protter \cite{Protter}) gave a proof of the inequality 
\begin{align}
\lambda_1(\Omega)\cdot \mathrm{IR}(\Omega)^2- \frac{\pi^2}{4}\geq \frac{(n- 1)\pi^2}{4}\cdot \Big(\frac{\mathrm{IR}(\Omega)}{\mathrm{diam}(\Omega)}\Big)^2. \label{MH's ineq}
\end{align}

For non-collapsing model with fixed volume, the sharp quantitative form of Faber-Krahn inequality is established in \cite{BDV}.

For convex domain $\Omega \subset \R^2$, Jerison \cite{Jer95} proved that the comparison inequality between the first two Dirichlet eigenvalues and the first two Dirichlet eigenvalues of corresponding weighted ODE, and some geometry invariant of the convex domain is involved in the inequalities. Also see \cite{Grieser} for a survey on the study of Dirichlet eigenvalues on thin tubes. 

For Neumann eigenvalue of convex domain $\Omega \subset \R^2$, Jerison \cite{Jer00} showed
the inequality between the first Neumann eigenvalues of $\Omega$ and the first Neumann eigenvalue of corresponding weighted ODE. His method is again the comparison to the ordinary differential equation, which originates from \cite{PW}.  Later Jerison's result is generalized to higher dimension in \cite{CJK} by discretizing the Neumann eigenfunction of Laplace operator. Zheng used the cross-sectional average of the eigenfunction, which can be viewed as continuous version of argument in \cite{CJK}, to develop the refined comparison in \cite{Zheng}. Our argument of Theorem \ref{thm NW ineq} follows this philosophy of comparison.

The first new thing we obtain is the sharp estimate of Neumann eigenvalue of the weighted ODE in term of the weight $h$, which is implied in the proof of Theorem \ref{thm HW-conj confirmed}. It is a consequence of Liouville transform and the estimate of the rest term $\int_0^d \frac{(h')^2}{h}(\zeta')^2$, where $\zeta$ is the corresponding Neumann eigenfunction of the weighted ODE. To obtain the estimate of the rest term, in Section \ref{sec weighted ODE}, we get the comparison estimate of $\zeta'(x)$ from the weighted ODE, which behaves like $\|x\|$ after detailed analysis. Then the estimate of $\int_0^d \frac{(h')^2}{h}(\zeta')^2$ is reduced to $\int_0^d \frac{(h')^2}{h}\|x\|^2$, which is obtained in Lemma \ref{lem major 2nd term}. The key Lemma \ref{lem major 2nd term} contributes the estimate of the major second term in the lower bound of Neumann eigenvalue, which relies on the fixed diameter and convex property of $\Omega$. 

The second technical tool we establish, is the estimate of the error term $\eta$ relating the PDE and weighted ODE, which is obtained in Section \ref{sec est of error}. This estimate needs the refined gradient estimate for Neumann eigenfunction of $\Omega$, which is discussed in Section \ref{sec refined grad est}. Then in Section \ref{sec comp of eigen and cross}, we use the estimate of $\eta$ and the estimate of weighted ODE, to control the difference between the modified cross-sectional average Laplace operator's first Neumann eigenfunction and the first Neumann eigenfunction of weighted ODE. 

Next we obtain a novel integral estimate of vertical (perpendicular to the axis corresponding to the weighted ODE) directional derivative of Neumann eigenfunction in Section \ref{sec dire derivative est}, which relies on the difference estimate in Section \ref{sec comp of eigen and cross}. All the above argument hold for $n$-dim $\Omega$.

Our first result which holds only for $2$-dim $\Omega$ is, the equations between width and projective width, the related boundary function equations after suitable choice of the $x$-axis. Those equations are established in Section \ref{sec key equality}.

In Section \ref{sec improved est} we get the improved integral estimate of $\eta\tilde{u}'$ for $2$-dim $\Omega$, which is almost the difference between Neumann eigenvalues of Laplace operator and the weighted ODE. They depends on the estimates in Section \ref{sec dire derivative est} and equations in Section \ref{sec key equality}. 

Finally in Section \ref{sec inequality}, we use Liouville transform and Cauchy-Schwartz inequality, to control the error terms appearing in the calculation of eigenvalue by Rayleigh quotient and prove the main theorem. 

\part{The estimates of PDE and ODE}

In the rest of the paper, for $v_0= (1, \mathbf{0})\in \mathbb{R}^n, \mathbf{0}\in \mathbb{R}^{n- 1}$, we assume 
\begin{align}
&\mathrm{diam}(\Omega)= 2,\quad \quad \max\{x\in \mathbb{R}: (x, y)\in \Omega\}= d, \quad \quad \min\{x\in \mathbb{R}: (x, y)\in \Omega\}= 0, \nonumber \\
&\Omega_x= \{(x, y)\in \Omega: y\in \mathbb{R}^{n- 1}\}, \quad \quad h(x)\vcentcolon= V(\Omega_x), \quad \quad \forall x\in [0, d], \nonumber \\
& h(\tau_0)= \max_{x\in [0, d]}h(x), \quad \quad   \tau_0\in (0, d)\subseteq (0, 2),\quad \quad \mathrm{diam}(\mathscr{P}_{v_0}(\Omega))= \mathscr{PW}(\Omega)= \epsilon. \nonumber 
\end{align}

\begin{definition}\label{def w-diameter}
{For $\Omega\subseteq \mathbb{R}^n$, up to the rigid motion of $\mathbb{R}^n$, we define the maximal $d> 0$ satisfying the above conditions, as the $w$-diameter of $\Omega$. 
}
\end{definition}

In the rest of the paper, we use $d$ to denote the $w$-diameter of $\Omega$. 
  
\section{The refined gradient estimates}\label{sec refined grad est}

\begin{notation}\label{notation norm of x}
{We define a special norm for $x\in [0, d]$ as follows:
\begin{align}
\|x\|\vcentcolon= \min\{x, d- x\}, \quad \quad \quad \forall x\in [0, d]. \label{def of a norm on interval} 
\end{align}
}
\end{notation}

For the first Neumann eigenfunction $u$ of $\Omega$ with the first Neumann eigenvalue $\mu_1(\Omega)$, after scaling, we can assume 
\begin{align}
-1\leq -k= \min u\leq u\leq 1= \max u, \label{bound of eigenfunction u}
\end{align}
and note $k\in (0, 1]$. By \cite{LY}, we know that
\begin{align}
\sup_{x\in \Omega}|Du|\leq \sqrt{\mu_1(\Omega)}\cdot \sqrt{1- u^2}\leq \sqrt{\mu_1(\Omega)}. \nonumber 
\end{align}

Recall the following formula on the second line of \cite[Page $5$]{HW}:
\begin{align}
\mu_1(\Omega)^{\frac{1}{2}}\cdot 2\geq \pi+ \frac{3}{4}\cdot \frac{(1- k)^2}{(1+ k)^2}\int_0^{\frac{\pi}{2}}\psi^2(\theta)d\theta,  \nonumber 
\end{align}
where $\psi(\theta)= \frac{\frac{4}{\pi}(\theta+ \cos\theta\sin\theta)- 2\sin\theta}{\cos^2\theta}$.

If $1- k\geq \epsilon$,  then $\displaystyle \mu_1(\Omega)\geq \frac{\pi^2}{4}+ C\cdot \epsilon^2$, we are done.

\textbf{In the rest argument,  we always assume that $\mu_1(\Omega)\leq \frac{\pi^2}{4}+ 1$ and 
\begin{align}
0\leq 1-k< \epsilon, \label{assumption on k}
\end{align}
where $\displaystyle -k= \min_{x\in \Omega}u(x)$.  And assume $\epsilon\leq \frac{1}{20}$}.

\begin{lemma}\label{lem grad est without k}
	{There is a universal constant $c> 0$,  such that if $\epsilon\in (0, c)$,  then there exists a universal constant $C_1> 0$ with
		\begin{align}
			\sup_{x\in [0,  \epsilon]}|Du|(x, y)\leq 2C_1^2\epsilon, \quad \quad \quad \sup_{x\in [0,  \epsilon]}u(x, y)\leq -k+ 4C_1^2\epsilon^2,  \nonumber \\
			\sup_{x\in [d- \epsilon,  d]}|Du|(x, y)\leq 2C_1^2\epsilon, \quad \quad \quad \inf_{x\in [d- \epsilon,  d]}u(x, y)\geq 1- 4C_1^2\epsilon^2. \nonumber
		\end{align}
	}
\end{lemma}

\pf
{\textbf{Step (1)}. It follows from \cite[Theorem $1.1$ (a)]{CJK} (also see \cite[Theorem $1.2$ (a)]{Jer00} for $n= 2$ case) and (\ref{bound of zero points}) of Lemma \ref{lem grad est of ODE}, there is a universal constant $c> 0$,  such that if $\epsilon\in (0, c)$,  then there is some $y_0\in \mathbb{R}$ such that
\begin{align}
u(d, y_0)> 0. \nonumber 
\end{align}

\textbf{Step (2)}.  From \cite{ZY}, let $v= \frac{u- \frac{1- k}{2}}{\frac{1+ k}{2}}$, we have 
	\begin{align}
		|Dv|\leq \sqrt{\mu_1(\Omega)}\sqrt{1- v^2}\sqrt{1+ \psi(\theta)}, \quad \quad \quad \theta= \sin^{-1}v\in [-\frac{\pi}{2}, \frac{\pi}{2}]; \nonumber 
	\end{align}
	where $\displaystyle \psi(t)= a\frac{\frac{4}{\pi}(t+ 2^{-1}\sin (2t))- 2\sin t}{\cos^2t}$ and $a= \frac{1- k}{1+ k}$. 
	
	Then note $\sqrt{\mu_1(\Omega)}\sqrt{1- u}\sqrt{1+ \psi(\theta)}\leq C_1$, where $C_1$ is a universal constant, then we get 
	\begin{align}
		|Du|\leq \sqrt{\mu_1(\Omega)}\sqrt{u+ k}\sqrt{1- u}\sqrt{1+ \psi(\theta)}\leq C_1\sqrt{u+ k}. \label{grad est we need}
	\end{align}
	
	Let $s_0= \max\{x\in [0,d]: \min_y u(x, y)= -k\}$, then by $|Du|\leq \sqrt{\mu_1(\Omega)}$, we have 
	\begin{align}
		\max_{y} u(s_0, y)\leq -k+ \sqrt{\mu_1(\Omega)}\epsilon< 0, \label{on s0 slice}
	\end{align}
	where the last inequality follows from (\ref{assumption on k}),  the bound of $\mu_1(\Omega)$ and $\epsilon\leq \frac{1}{20}$. 
	
	Now from (\ref{on s0 slice}) we have
	\begin{align}
		|Du|(s_0, y)\leq C_1\sqrt{u(s_0, y)+ k} \leq C_1\sqrt{\sqrt{\mu_1(\Omega)}\epsilon}= C_1\mu_1(\Omega)^{\frac{1}{4}} 
		(\epsilon)^{\frac{1}{2}}; \nonumber 
	\end{align}
	therefore we get
	\begin{align}
		u(s_0, y)\leq -k+ \max_z |Du|(s_0, z)\cdot \epsilon= -k+ C_1\mu_1(\Omega)^{\frac{1}{4}}(\epsilon)^{\frac{3}{2}}. \label{improved C0 est} 
	\end{align}
	
	Using (\ref{grad est we need}) and (\ref{improved C0 est}) to estimate $Du(s_0, y)$; then by induction, we get
	\begin{align}
		&|Du(s_0, y)|\leq C_1^{1+ 2^{-1}+ \cdots + 2^{-(m- 1)}}\mu_1(\Omega)^{2^{-(m+ 1)}}(\epsilon)^{2^{-1}+\cdots + 2^{-m}}. \nonumber
	\end{align}
	Let $m\rightarrow\infty$, we obtain
	\begin{align}
		|Du|(s_0, y)\leq C_1^2\epsilon. \label{C1 est on slice best}
	\end{align} 
	
	Now we have 
	\begin{align}
		u(s_0, y)\leq -k+ \sup_z |Du|(s_0, z)\cdot \epsilon\leq -k+ C_1^2\epsilon^2. \label{C0 est on slice best} 
	\end{align}
	
\textbf{Step (3)}. If $s_0> \epsilon$, define 
\begin{align}
\Omega_1= \Omega\cap \{x< s_0\},  \quad \quad \quad u(x_0)= \max_{x\in \overline{\Omega_1}}u(x).  \nonumber 
\end{align}

If $u(x_0)\geq 0$,  then $x_0\notin \{(s_0, y)\in \Omega\}$.  Define $q= (s_0, 0),  \tilde{q}= (d, y_0)$,  then
\begin{align}
u^{-1}(0)\cap \overline{x_0, q}\neq \emptyset,  \quad \quad u^{-1}(0)\cap \overline{q, \tilde{q}}\neq \emptyset.  \nonumber 
\end{align}

It is well-known that $u^{-1}(0)$ is connected,  choose 
\begin{align}
z_1\in u^{-1}(0)\cap \overline{x_0, q},  \quad \quad z_2\in u^{-1}(0)\cap \overline{q, \tilde{q}};  \nonumber 
\end{align}
we get that $u^{-1}(0)\cap \{(s_0, y)\in \Omega\}\neq \emptyset$,  which contradicts (\ref{on s0 slice}).  

Hence  $u(x_0)< 0$.
	
Then $\Delta u= -\mu_1 u> 0$ on $\Omega_1$. If $x_0\in \Omega_1$, then $\Delta u(x_0)\leq 0$, which is the contradiction.
	
	If $x_0\in \mathrm{Interior}(\partial\Omega_1\cap \partial\Omega)$, then by the Strong Maximum Principle, we get $\frac{\partial u}{\partial\vec{n}}(x_0)> 0$, which is the contradiction to the Neumann boundary condition on $u$. 
	
	Hence $x_0\in \Omega\cap \{x= s_0\}$, we get 
	\begin{align}
		\max_{x\in \overline{\Omega_1}}u(x)= u(x_0)\leq \max_{y} u(s_0, y)\leq -k+ \sqrt{\mu_1(\Omega)}\epsilon, \label{max is close to -k too} 
	\end{align}
	 
Therefore we have
	\begin{align}
		\sup_{x\in [0, \epsilon]}u(x, y)\leq u(s_0, y)\leq -k+ C_1^2\epsilon^2. \nonumber 
	\end{align}

Now by (\ref{grad est we need}) again, we obtain
	\begin{align}
		\sup_{x\in [0,  \epsilon]}|Du|(x, y)\leq C_1^2\epsilon. \nonumber 
	\end{align}
	
	If $s_0\leq \epsilon$,  assume $u(s_0, y_0)= -k$, we have 
	\begin{align}
		u(x, y)\leq u(s_0, y_0)+ \sup_{x\in [0, \epsilon]}|Du|(x, z)\cdot 2\epsilon= -k+ 2\epsilon\cdot \sup_{x\in [0, \epsilon]}|Du|(x, z).  \nonumber 
	\end{align}
	
	Then we obtain
	\begin{align}
		&|Du|(x, y)\leq C_1\sqrt{u+ k}\leq C_1\cdot \sqrt{2\epsilon\cdot \sup_{x\in [0, \epsilon]} |Du|}, \nonumber
		\end{align}
		
	Taking sup for $x\in [0, \epsilon]$ on the left side of the above inequality,  	we get $\displaystyle \sup_{x\in [0, \epsilon]}|Du|(x, y)\leq 2C_1^2\epsilon$.
		Now 
		\begin{align}
		&u(x, y)\leq -k+ 2\epsilon\cdot \sup |Du|.\leq -k+ 4C_1^2\epsilon^2 .  \nonumber 
	\end{align}
	
	The inequality on $[d- \epsilon,  d]$ is obtained similarly.   
}
\qed

Now we have the following gradient estimates for any dimension.
\begin{prop}\label{prop directional deri est-any dim}
	{There is a universal constant $c> 0$,  such that if $\epsilon\in (0, c)$,  then we can find a universal constant $C> 0$ with
		\begin{align}
			|Du(x, y)|\leq C\cdot \max\{\epsilon, \|x\|\},  \quad \quad \forall (x, y)\in \Omega .  \label{grad est in term of x}
		\end{align}
	}
\end{prop}

\pf
{\textbf{Step (1)}.  If $|x|\le \epsilon$ or $|x-d|\le \epsilon$,  by Lemma \ref{lem grad est without k}, we have $|Du(x, y)|\leq C\cdot \epsilon$.

Without loss of generality, we only need to show that for $x\in [\epsilon,d/2]$, $|Du(x, y)|\leq C|x|$.
	
Choose $m(\epsilon)\in \mathbb{N}$, such that $2^{-m-1}\le \epsilon \le 2^{-m}$. For $n=0, 1, \cdots,  m$, let
\[|u|_n=\sup_{x\in [2^{-n-1}, 2^{-n}]}|u|,\quad\quad |Du|_n=\sup_{x\in [2^{-n-1}, 2^{-n}]}|Du|.\] 
In Lemma \ref{lem grad est without k}, we have
$\sup_{x\in [0, \epsilon]}u(x, y)\leq -k+ 4C_1^2\epsilon^2$. 
	\begin{align}\label{estimate of |u|_m}
		|u|_m\le -k+ 4C_1^2\epsilon^2+(2^{-m-1}+ \epsilon)|Du|_m\le -k+4C_1^22^{-2m}+ 2\cdot 2^{-m}|Du|_m.
	\end{align}
	By (\ref{grad est we need}),
	\begin{align}
		&|Du(x,y)|\le C_1\sqrt{u+k}\le C_1\sqrt{4C_1^22^{-2m}+ 2\cdot 2^{-m}|Du|_m}.\nonumber
	\end{align}
	Hence we get $|Du|_m\le 5C_1^22^{-m}$ . By (\ref{estimate of |u|_m}), $|u|_m\le -k+25C_1^22^{-2m}$.
	
\textbf{Step (2)}.	We prove by induction that $|u|_n\le -k+25C_1^22^{-2n} , |Du|_n\le 5C_1^22^{-n}$. 
	
	If $|u|_n\le -k+25C_1^22^{-2n} , |Du|_n\le 5C_1^22^{-n}$,  note $C_1\geq 1$,  we get 
	\begin{align}
		&|u|_{n-1}\le |u|_n+(2^{-n}+2\epsilon)|Du|_{n-1}\le |u|_n+3\cdot 2^{-n}|Du|_{n-1}\nonumber\\
		&|Du|_{n-1}\le C_1\sqrt{k+|u|_{n-1}}\le C_1\sqrt{25C_1^22^{-2n}+3\cdot 2^{-n}|Du|_{n-1}}.  \nonumber
	\end{align}
	Thus $|Du|_{n-1}\le 5C_1^22^{-(n-1)}$ and 
	\begin{align}
	|u|_{n-1}\le |u|_n+ 3\cdot 2^{-n}|Du|_{n-1}\leq -k+ 25C_1^2 2^{-2(n- 1)}
.  \nonumber 
\end{align}
	
	Note that for $x\in [2^{-n-1}, 2^{-n}]$,  we have $2^{-n}\leq 2x$.  Then we get
	\begin{align}
	|Du|(x)\le |Du|_n\leq 10C_1^2|x|.  \nonumber 
	\end{align}
	From the above we get the conclusion for $x\in [0, \frac{d}{2}]$. 
	
	Similar argument yields the conclusion for $x\in [\frac{d}{2},  d]$.  
	
	Now there is a universal constant $C_1> 0$ such that  
		\begin{align}
			|Du(x, y)|\leq 10C_1^2\cdot \max\{\epsilon, \|x\|\},  \quad \quad \forall (x, y)\in \Omega .  \nonumber 
		\end{align}
		}
		\qed
	
\section{The weighted ODE}\label{sec weighted ODE}

Define $H=\{f\in W^{1,2}(0,d)| \int_0^d hf=0\}$, and for $f\in H$, we define the \textbf{Rayleigh quotient of $f$} as 
\begin{align}
	\mathscr{R}(f)= \frac{\int_0^d h(f')^2}{\int_{0}^d hf^2}. \nonumber 
\end{align}

\begin{definition}\label{def mu1 N}
	{We define
		\begin{align}
			\mu_1(N)\vcentcolon= \inf_{f\in H- \{0\}}\mathscr{R}(f).\nonumber
		\end{align}
	}
\end{definition}

\begin{theorem}\label{thm existence of phi to ODE}
	{There is a unique function $\phi\in C^1[0, d]\cap C^\infty(0, d)$,  such that 
		\begin{equation}\label{height ODE}
			\left\{
			\begin{array}{rl}
				&-(h\phi')'= \mu_1(N)\cdot h\phi\ ,  \quad \quad \quad \quad \quad\quad \quad x\in (0, d), \\
				&\phi(0)= -1,  \quad \quad \phi'(0)= \phi'(d)= 0 \ ; \quad \quad \quad \quad \quad\quad \quad 
			\end{array} \right.
		\end{equation}
	}
\end{theorem}

\pf
{It follows from Lemma \ref{lem existence of minimizer-new} and the regularity theory of weak solution to ODE.
}
\qed

\begin{notation}\label{notation phi}
	{In the rest of the paper,  unless otherwise mentioned,  the function $\phi$ is the first Neumann eigenfunction of $N= ([0, d],  hdx)$ with respect to the first Neumann eigenvalue $\mu_1(N)$ and $\phi(0)= -1$.
	}
\end{notation}

From \cite[Page $409$]{AS},  we know that $j_{0, 1}\leq 2.5$,  where $j_{0, 1}$ is the first positive zero of the Bessel function $J_0$ of order $0$.  By \cite[Proposition $2$]{Kroger-upper},  we get 		
\begin{align}
	\frac{\pi^2}{4}\leq \mu_1(N)\leq (2j_{0, 1})^2\leq 25.  \label{uniform upper bound of mu 1 N}
\end{align}

\begin{lemma}\label{lem grad est of ODE}
	{There is a universal constant $C= C(n)> 0$,  assume $\phi(x_0)= 0$,  then 
		\begin{align}
			&\phi'(x)\geq C^{-1}\cdot \|x\|,  \quad \quad \quad \forall \|x\|\in [0,  \frac{1}{10^n}],   \nonumber \\
			& x_0\in [\frac{1}{10^n},  d- \frac{1}{10^n}],  \label{bound of zero points} \\
			&\sup_{x\in [0, d]}|\phi(x)|\leq C,  \label{uniform bound of phi}\\
			&0\leq \phi'(x)\leq C\cdot \|x\|, \quad \quad \forall x\in [0, d]. \label{upper bound of phi' in x}
		\end{align}
	}
\end{lemma}		

\pf
{\textbf{Step (1)}. Note $\phi'\geq 0$ is from the existence of the solution $\phi$. In the rest argument, we use $\mu_1$ to denote $\mu_1(N)$ for simplicity.
	
Note that $h^{\frac{1}{n-1}}$ is concave, for $t<x$, we have $h^{\frac{1}{n-1}}(x)\ge \frac{d-x}{d-t}\cdot h^{\frac{1}{n-1}}(t)+ \frac{x-t}{d-t}\cdot h^{\frac{1}{n-1}}(d)\ge \frac{d-x}{d-t}\cdot h^{\frac{1}{n-1}}(t)$. Thus
		\begin{align}
			\frac{h(t)}{h(x)}\le (\frac{d- t}{d- x})^{n-1}\label{Concave property of h^{1/(n-1)}}
	\end{align}
	
	Integrating (\ref{height ODE}) from $0$ to $x$, using $\phi(t)\geq \phi(0)= -1$ and (\ref{Concave property of h^{1/(n-1)}}) we get
	\begin{align}
		\phi'(x)&= -\frac{1}{h}\int_0^x\mu_1 h(t)\phi(t)dt\leq \mu_1\int_0^x \frac{h(t)}{h(x)}dt \nonumber \\
		&\leq \mu_1\int_0^x (\frac{d- t}{d- x})^{n-1}dt= \mu_1\cdot \frac{d^n-(d-x)^n}{n(d- x)^{n-1}}. \nonumber 
	\end{align}

Note $d^n-(d-x)^n\le nx d^{n-1}$, now for $x\in [0, \frac{d}{2}]$, we obtain
	\begin{align}
		\phi'(x)&\leq  2^n\mu_1\cdot x, \quad \quad \forall x\in [0, \frac{d}{2}] \label{upper bound phi' at 0 end}
	\end{align}
	
	By similar argument as above, for $x\in [\frac{d}{2}, d]$, we also have
	\begin{align}
		\phi'(x)&\leq 2^n\mu_1\phi(d)\cdot (d- x), \quad \quad \forall x\in [\frac{d}{2}, d]. \label{upper bound of phi' at d end}
	\end{align}
	
	\textbf{Step (2)}.  Note $[d- \frac{1}{\sqrt{2^n\mu_1}},  d]\subseteq [\frac{d}{2},  d]$ by (\ref{uniform upper bound of mu 1 N}).  From (\ref{upper bound of phi' at d end}),   for any $x\in [d- \frac{1}{\sqrt{2^n\mu_1}},  d]\subseteq [\frac{d}{2},  d]$;  we get
	\begin{align}
		\phi(d)- \phi(x)= \int_x^d \phi'\leq 2^n\int_{d- \frac{1}{\sqrt{2^n\mu_1}}}^d \mu_1\phi(d)(d- x)dx\leq \frac{1}{2}\phi(d).  \nonumber 
	\end{align}
	
	And we get 
	\begin{align}
		\phi(x)\geq \frac{1}{2}\phi(d)> 0,  \quad \quad \forall x\in  [d- \frac{1}{\sqrt{2^n\mu_1}},  d]\subseteq [\frac{d}{2},  d].  \label{lower bound of phi x near d end}
	\end{align}
	
	By (\ref{lower bound of phi x near d end}) and (\ref{uniform upper bound of mu 1 N}),  we get that
	\begin{align}
	x_0\leq d- \frac{1}{\sqrt{2^n\mu_1}}\leq d- \frac{1}{10^n}.  \label{upper bound of x_0}
	\end{align}
	
	On the other hand,  using (\ref{upper bound phi' at 0 end}),  similarly we have
	\begin{align}
		\phi(\frac{1}{\sqrt{2^n\mu_1}})- \phi(0)= \int_0^{\frac{1}{\sqrt{2^n\mu_1}}}\phi'(x)dx\leq \frac{1}{2}.  \nonumber 
	\end{align}
	
	Then $\phi(\frac{1}{\sqrt{2^n\mu_1}})\leq -\frac{1}{2}< 0$,  we get
	\begin{align}
		x_0\geq \frac{1}{\sqrt{2^n\mu_1}}\geq \frac{1}{10^n}. \label{lower bound of x_0}
	\end{align}
	
	By (\ref{upper bound of x_0}) and (\ref{lower bound of x_0}),  we get (\ref{bound of zero points}).
	
	\textbf{Step (3)}.  By (\ref{upper bound of x_0}) and (\ref{lower bound of phi x near d end}),  using the concave property of $h^{\frac{1}{n- 1}}$, we have
	\begin{align}
		CV(\Omega) &\geq -\int_0^{x_0} h\phi= \int_{x_0}^d h\phi\geq \int_{d- \frac{1}{10^n}}^d h\phi\geq \frac{1}{2}\phi(d)\int_{d- \frac{1}{10^n}}^d h \nonumber \\
		&\geq \frac{1}{2}\phi(d)\cdot \frac{[h^{\frac{1}{n-1}}(d)+ h^{\frac{1}{n-1}}(d- \frac{1}{10^n})]^{n-1}}{2^{n-1}}\cdot \frac{1}{2\cdot 10^n}  \nonumber \\
		&\geq CV(\Omega)\cdot \phi(d).   \nonumber 
	\end{align}
	
	Hence $\phi(d)\leq C$,  then 
	\begin{align}
		\sup_{x\in [0, d]}|\phi(x)|\leq \phi(d)+ |\phi(0)|\leq C.  \nonumber 
	\end{align}
	This implies (\ref{uniform bound of phi}).
	
	By (\ref{upper bound phi' at 0 end}),  (\ref{upper bound of phi' at d end}) and (\ref{uniform bound of phi}), we get (\ref{upper bound of phi' in x}). 
	
	\textbf{Step (4)}. Note $h^{\frac{1}{n-1}}$ is concave by Brunn-Minkowski inequality and the convex property of $\Omega$. Then for $0<t<x$, we have $h^{\frac{1}{n-1}}(t)\ge \frac{t}{x}h^{\frac{1}{n-1}}(x)+\frac{x-t}{x}h^{\frac{1}{n-1}}(0)\ge \frac{t}{x}h^{\frac{1}{n-1}}(x)$. Also from (\ref{height ODE}), using (\ref{upper bound phi' at 0 end}), for $x\in [0,  \frac{1}{10^n}]$; we have
	\begin{align}
		\phi'(x)&= -\frac{1}{h}\int_0^x\mu_1 h(t)[-1+ \int_0^t \phi'(s)ds]dt \nonumber \\
		&\geq \mu_1\int_0^x \frac{h(t)}{h(x)}dt- \frac{1}{h(x)}\int_0^x \mu_1 h\cdot [\int_0^t 2^n\mu_1 sds]dt \nonumber \\
		&\geq \mu_1\int_0^x \frac{h(t)}{h(x)}dt- 2^{n-1}\mu_1^2\int_0^x \frac{h(t)}{h(x)}\cdot t^2 dt \nonumber \\
		&= \mu_1\int_0^x (1- \mu_1t^2)\frac{h(t)}{h(x)}dt \geq \frac{1}{2}\mu_1\int_0^x (\frac{t}{x})^{n-1} dt \nonumber \\
		&\geq \mu_1\frac{x}{2n} .\nonumber
	\end{align}
	
	Then
	\begin{align}
		\phi'(x)\geq Cx, \quad \quad \quad \forall x\in [0, \frac{1}{10^n}]. \label{lower bound of phi' at 0 end}
	\end{align}
	
	Using (\ref{upper bound of phi' at d end}) as above,  for $x\in [d- \frac{1}{10^n}, d]$, we have
	\begin{align}
		\phi'(x)&= \frac{1}{h}\int_x^d\mu_1 h(t)[\phi(d)- \int_t^d \phi'(s)ds]dt\geq \mu_1\phi(d)\frac{d- x}{2n}. \label{lower bound of phi' at d end} 
	\end{align}
	
	\textbf{Step (5)}. Now for $x\in [0,\frac{d}{2}]$, from (\ref{upper bound phi' at 0 end}), we get
	\begin{align}
		\phi(x)= -1+ \int_0^x \phi'(t)dt\leq -1+ 2^{n-1}\mu_1x^2. \nonumber 
	\end{align} 
	
	Therefore note $\sqrt{\frac{1}{2^n\mu_1}}\leq \min\{\frac{d}{2}, x_0\}$, we have
		\begin{align}
			0&= \int_0^d h\phi\leq \int_0^{\sqrt{\frac{1}{2^n\mu_1}}} h\cdot \phi(x)dx+\int_{x_0}^d h\cdot \phi(d) \nonumber \\
			&\leq \int_0^{\sqrt{\frac{1}{2^n\mu_1}}} -h(x)\cdot (1- 2^{n-1}\mu_1x^2) dx+ V(\Omega)\phi(d) \nonumber \\
			&\leq -\frac{1}{2}\int_0^{\sqrt{\frac{1}{2^n\mu_1}}} h(x)dx+ V(\Omega)\phi(d) \nonumber \\
			&\leq -\frac{1}{2^{n+1}}h\Big(\sqrt{\frac{1}{2^n\mu_1}}\Big)\cdot \sqrt{\frac{1}{2^n\mu_1}}+ V(\Omega)\phi(d) \nonumber \\
			&\leq -C(n)V(\Omega)+ V(\Omega)\phi(d) .\nonumber 
	\end{align}
	
	Simplifying the above yields
	\begin{align}
		\phi(d)\geq  C(n). \label{lower bound of phi d}
	\end{align}
	
	Plugging (\ref{lower bound of phi d}) into (\ref{lower bound of phi' at d end}),  we obtain
	\begin{align}
		\phi'(x)\geq C\cdot (d- x), \quad \quad \quad \forall x\in [d- \frac{1}{10^n}, d].\label{lower bound of phi' near d end}
	\end{align}
	
	By (\ref{lower bound of phi' at 0 end}) and (\ref{lower bound of phi' near d end}),  the conclusion follows.
}
\qed		

\begin{lemma}\label{lem L2 upper bound deri of phi}
	{There is a universal constant $C> 0$ such that 
		\begin{align}
			& 	C^{-1}\cdot V(\Omega)\leq \int_0^d h\cdot (\phi)^2\leq C\cdot V(\Omega),\label{integra bound of phi} \\
			&C^{-1}\cdot V(\Omega)\leq \int_0^d h\cdot (\phi')^2\leq C\cdot V(\Omega).\label{integra bound of phi'}
		\end{align}
	}
\end{lemma}

\pf
{Using Lemma \ref{lem grad est of ODE},  note $h\leq \frac{V(\Omega)}{\frac{d}{n}}\leq C\cdot V(\Omega)$,  we get
	\begin{align}
		\int_0^d h\cdot (\phi)^2\leq C\cdot V(\Omega).  \nonumber 
	\end{align} 
	
	On the other hand,  for any $x\in [0,  \frac{1}{10^n}]$,  by Lemma \ref{lem grad est of ODE} we have 
	\begin{align}
		\phi(x)= \phi(0)+ \int_0^x \phi'\leq-1+ \int_0^x 2^n\mu_1tdt\leq -\frac{1}{2}.  \nonumber 
	\end{align}
	
	Now we have
	\begin{align}
			\int_0^d h\phi^2&\geq \int_0^{10^{-n}} h\phi^2\geq \frac{1}{4}\int_0^{10^{-n}} h\geq \frac{1}{8}\cdot \frac{h(\frac{1}{10^n})}{10^{n}\cdot 2^{n-1}}\geq CV(\Omega).  \nonumber 
	\end{align}
	
	By the above,  (\ref{integra bound of phi}) is obtained.
	
	Now from (\ref{integra bound of phi}) and (\ref{uniform upper bound of mu 1 N}), we get 
	\begin{align}
		\int_0^d h\cdot (\phi')^2&= \mu_1(N)\int_0^d h\phi^2\in [C_1V(\Omega), C_2V(\Omega)].  \nonumber 
	\end{align}
}
\qed

\section{The estimate of error term from Liouville's transform}

\begin{definition}\label{def intric width}
{Assume $t_0\in [0, d]$ such that 
\begin{align}
\mathrm{diam}(\Omega_{t_0})= \max_{t\in [0, d]}\mathrm{diam}(\Omega_{t_0})= \hat{\epsilon}. \nonumber 
\end{align}
}
\end{definition}

The following Lemma gives the crucial lower bound of the geometric error term for $\Omega$ with almost maximal $w$-diameter.

\begin{lemma}\label{lem major 2nd term}
	{There is a universal constant $C= C(n)> 0$, such that if $d\geq 2- 10^{-6n}\cdot \hat{\epsilon}^2$, then
		\begin{align}
			\int_0^{10^{-1}} \frac{|h'(x) |^2}{h(x) }\|x\|^2dx \geq C\cdot V(\Omega)\cdot \hat{\epsilon}^2. \nonumber 
		\end{align}
	}
\end{lemma}

\pf
{\textbf{Step (1)}. After rotation of the coordinates $\{y_i\}_{i= 1}^{n- 1}$ perpendicular to $x$-axix, we can get 
\begin{align}
\epsilon_1(t_0)\geq \frac{1}{2}\hat{\ep}, \nonumber 
\end{align}
where $\epsilon_1(\cdot)$ is defined as follows:
\begin{align}
\epsilon_1(x)\vcentcolon= \sup_{(x,y_1,...,y_{n-1}), (x,z_1,...,z_{n-1})\in \Omega_x}|y_1- z_1|. \nonumber 
\end{align}

Let $\tau_1= 10^{-6n}\hat{\epsilon}^2$. By $\Omega$ is convex, we get
\begin{align}
\epsilon_1(\frac{t_0+ \tau_1}{2})\geq \frac{1}{2}\{0+ \frac{1}{2}\mathrm{diam}(\Omega_{t_0})\}\geq \frac{\hat{\epsilon}}{4}. \label{epsilon 1 lower bound}
\end{align}	
	
Since $\mathrm{diam}(\Omega)=2$, we get
\begin{align}
		[\frac{1}{2}\epsilon_1(\tau_1)]^2&\leq 4- (d-\tau_1)^2\leq (2- d+ \tau_1)(2+ d- \tau_1) \leq 8\tau_1 . \nonumber 
\end{align}
Then we get
\begin{align}
	\epsilon_1(\tau_1)\leq  \frac{8\hat{\epsilon}}{10^{3n}}. \label{h epsi 1.5 is small epi-general}
\end{align}

\textbf{Step (2)}. For $s\in \mathbb{R}$ and $e\in \mathbb{R}^{n- 1}$, we define
\begin{align}
P(e, s)= \{y\in \mathbb{R}^{n- 1}: (y- s\cdot e)\cdot e= 0\}. \nonumber 
\end{align}

There is $s_1\in \mathbb{R}$ such that
\begin{align}
\mathcal{H}^{n- 2}(P(e_1, s_1)\cap \check{\Omega}_{\tau_1})= \max_{s\in \mathbb{R}}\mathcal{H}^{n- 2}(P(e_1, s)\cap \check{\Omega}_{\tau_1}). \nonumber 
\end{align}


Note $\frac{1}{2}[P(e_1, s_1)\cap \check{\Omega}_{t_0}+ P(e_1, s_1)\cap \check{\Omega}_{\tau_1}]\subseteq [P(e_1, s_1)\cap \check{\Omega}_{\frac{t_0+ \tau_1}{2}}]$ by $\Omega$ is convex. Because $\frac{1}{2}[P(e_1, s_1)\cap \check{\Omega}_{t_0}]), \frac{1}{2}[P(e_1, s_1)\cap \check{\Omega}_{\tau_1}]$ are convex subsets of $\mathbb{R}^{n- 1}$, from Brunn-Minkowski's inequality, we get
\begin{align}
&\quad \mathcal{H}^{n- 2}(P(e_1, s_1)\cap \check{\Omega}_{\frac{t_0+ \tau_1}{2}})\geq \mathcal{H}^{n- 2}(\frac{1}{2}[P(e_1, s_1)\cap \check{\Omega}_{t_0}+ P(e_1, s_1)\cap \check{\Omega}_{\tau_1}])\nonumber\\
&\geq \Big\{\Big(\mathcal{H}^{n- 2}(\frac{1}{2}[P(e_1, s_1)\cap \check{\Omega}_{t_0}])\Big)^{\frac{1}{n- 2}}+ \Big(\mathcal{H}^{n- 2}(\frac{1}{2}[P(e_1, s_1)\cap \check{\Omega}_{\tau_1}])\Big)^{\frac{1}{n- 2}}\Big\}^{n- 2}\nonumber\\
&\geq \frac{1}{2^{n- 2}}\mathcal{H}^{n- 2}(P(e_1, s_1)\cap \check{\Omega}_{\tau_1}). \label{n-2 dim measure lower bound} 
\end{align}

Note that $\check{\Omega}_{\frac{t_0+ \tau_1}{2}}$ is convex, hence by (\ref{epsilon 1 lower bound}) and (\ref{n-2 dim measure lower bound}), we get
\begin{align}
&\quad \max_{x\in [0, d]}h(x)\geq \mathcal{H}^{n- 1}(\check{\Omega}_{\frac{t_0+ \tau_1}{2}})\nonumber \\
&\geq \frac{1}{2}\epsilon_1(\frac{t_0+ \tau_1}{2})\cdot \max_{s\in \mathbb{R}} \mathcal{H}^{n- 2}(P(e_1, s)\cap \check{\Omega}_{\frac{t_0+ \tau_1}{2}}) \nonumber \\
&\geq \frac{\hat{\epsilon}}{8}\cdot \mathcal{H}^{n- 2}(P(e_1, s_1)\cap \check{\Omega}_{\frac{t_0+ \tau_1}{2}})\geq \frac{\hat{\epsilon}}{2^{n+ 1}}\mathcal{H}^{n- 2}(P(e_1, s_1)\cap \check{\Omega}_{\tau_1}) \nonumber \\
&\geq 10^{2n} \epsilon_1(\tau_1)\cdot \max_{s\in \mathbb{R}}\mathcal{H}^{n- 2}(P(e_1, s)\cap \check{\Omega}_{\tau_1})\nonumber \\
&\geq 10^{2n}\mathcal{H}^{n- 1}(\check{\Omega}_{\tau_1})= 10^{2n}\cdot h(\tau_1). \label{upper bound h at small x}
\end{align}

\textbf{Step (3)}. Using the concavity of $h^{\frac{1}{n-1}}$, we have 
\begin{align}
		h^{\frac{1}{n-1}}(\frac{1}{10})\geq \frac{1}{20}\max_{x\in [0, d]}h^{\frac{1}{n-1}}(x).\label{h d over 2-general}
\end{align}
	
	Using (\ref{h d over 2-general}), (\ref{upper bound h at small x}) and H\"older inequality, we obtain
	\begin{align}
		&\int_{\tau_1}^{\frac{1}{10}}\frac{|h'(x) |^2}{h(x) }\|x\|^2dx\geq   \frac{(\int_{\tau_1}^{\frac{1}{10}}|h\cdot h'|dx)^2}{\int_{\tau_1}^{\frac{1}{10}}\frac{h^3}{x^2}dx} 
		\geq \frac{|h^2(\frac{1}{10})- h^2(\tau_1)|^2}{\max\limits_{x\in [0, d]}h^3(x)\cdot \int_{\tau_1}^{\frac{1}{10}}\frac{1}{x^2}dx}
	\nonumber \\
		&\geq C\frac{\max_{x\in [0, d]}h(x)}{\frac{1}{\tau_1}} \geq C\hat{\epsilon}^2\max_{x\in [0, d]}h(x)\geq CV(\Omega)\cdot \hat{\epsilon}^2. \nonumber 
	\end{align}
}
\qed

\part{The general convex domain}

\section{The estimate of error terms between PDE and ODE}\label{sec est of error}

For $x\in [0, d]$, define $\bar{u}(x)=  \frac{1}{h(x)}\int_{\Omega_x} u(x, y)dy$. Furthermore we define the \textbf{modified cross-sectional average} $\tilde{u}$ and the \textbf{error term} $\eta$ as follows:
\begin{align}
	\hat{u}&=\begin{cases}
		\bar{u}(\epsilon), \quad x\in [0,\epsilon]\\
		\bar{u}(x), \quad x\in [\epsilon,d-\epsilon]\\
		\bar{u}(d-\epsilon), \quad x\in [d-\epsilon,d]\\
	\end{cases}\nonumber \\
c_1&=\frac{\int_{0}^d h\hat{u}}{\int_{0}^d h},\quad \tilde{u}= \hat{u}- \frac{\int_{0}^d h\hat{u}}{\int_{0}^d h},\label{def of tilde u} \\
\eta(x)&= h\tilde{u}'(x)+ \mu_1(\Omega)\int_0^x h\tilde{u}ds. \label{def of eta}
\end{align}

\begin{remark}\label{rem def of hat u}
{The definition of $\hat{u}$ near the end points of $[0, d]$, is similar to  \cite[Page $5134$, (iii) and (iv)]{CJK}. The definition of $\hat{u}$ can be viewed as the combination of the definition of cross-sectional average $\bar{u}$ defined in \cite{Zheng} and discrete cross-sectional average $\tilde{\phi}$ defined in \cite[Section $5$]{CJK}.
}
\end{remark}

For $s\in \mathbb{R}$, we define
\begin{align}
\check{\Omega}_s\vcentcolon= \{y\in \mathbb{R}^{n- 1}: (s, y)\in \Omega_s\}. \nonumber 
\end{align}

A similar result is showed in \cite[Lemma $6$]{Zheng}, but the definition of $\eta$ here is a little different from \cite{Zheng}. 
\begin{lemma}\label{lem rough est of eta for general dim}
{There is a universal constant $C> 0$, such that for any $x\in [\epsilon, d- \epsilon]$,
\begin{align}
&|\eta|(x)\leq Ch(x)\cdot \epsilon , \quad \quad |\tilde{u}'- \frac{1}{h}\int_{\Omega_x}u_x|\leq C\epsilon. \nonumber  
\end{align}
}
\end{lemma}

\pf
{\textbf{Step (1)}.  In the rest argument, we assume $\mu_1= \mu_1(\Omega)$. Assume $(x, y)\in [0, d]\times \mathbb{R}^{n- 1}$. Now 
\begin{align}
\eta(x)= h\tilde{u}'+ \mu_1\int_0^x h\tilde{u}. \nonumber 
\end{align}

For $x\in [\epsilon, d- \epsilon]$, we have 
\begin{align}
\tilde{u}'(x)= \bar{u}'(x)= \lim_{x_1\rightarrow x^+} \frac{\bar{u}(x_1)- \bar{u}(x)}{x_1- x}. \label{deri of bar u firstly}
\end{align}

Next assume $(0, y_0)\in \partial\Omega$ and $x_1> x$, we define $T: \check{\Omega}_{x_1}\rightarrow \check{\Omega}_{x_0}$ as follows
\begin{align}
T(y)= y_0+ \frac{x}{x_1}(y- y_0).\nonumber 
\end{align}
Note we have $T\hat{\Omega}_{x_1}\subseteq \check{\Omega}_{x}$ and $\mathrm{Jac}(T)= (\frac{x}{x_1})^{n- 1}$. 

Direct computation yields
\begin{align}
&\frac{\bar{u}(x_1)- \bar{u}(x)}{x_1- x}= \frac{1}{h(x_1)}\int_{\Omega_{x_1}} \frac{u(x_1, y)- \bar{u}(x)}{x_1- x}dy \nonumber \\
&= \frac{\mathrm{Jac}(T^{-1})}{h(x_1)}\int_{T\check{\Omega}_{x_1}} \frac{u(x_1, T^{-1}z)- \bar{u}(x)}{x_1- x}dz \nonumber \\
&= \frac{\mathrm{Jac}(T^{-1})}{h(x_1)}\Big\{\int_{T\check{\Omega}_{x_1}} \frac{u(x_1, T^{-1}z)- u(x, z)}{x_1- x}dz+ \int_{T\check{\Omega}_{x_1}} \frac{u(x, z)- \bar{u}(x)}{x_1- x}dz\Big\}. \nonumber 
\end{align}

Define 
\begin{align}
(I)\vcentcolon&=  \frac{\mathrm{Jac}(T^{-1})}{h(x_1)}\int_{T\check{\Omega}_{x_1}} \frac{u(x_1, T^{-1}z)- u(x, z)}{x_1- x}dz \nonumber \\
(II)\vcentcolon&= \frac{\mathrm{Jac}(T^{-1})}{h(x_1)}\int_{T\check{\Omega}_{x_1}} \frac{u(x, z)- \bar{u}(x)}{x_1- x}dz. \nonumber 
\end{align}

\textbf{Step (2)}. Note $T^{-1}(z)= y_0+ \frac{x_1}{x}(z- y_0)$. We firstly have
\begin{align}
&\lim_{x_1\rightarrow x^+} (I)= \frac{1}{h(x)}\int_{\check{\Omega}_{x}} \lim_{x_1\rightarrow x^+}\frac{u(x_1, T^{-1}z)- u(x, z)}{x_1- x}dz \nonumber \\
&= \frac{1}{h(x)}\int_{\check{\Omega}_{x}} u_x(x, z)+ (D_yu\cdot \frac{z- y_0}{x}) dz. \label{first est of I}
\end{align}

Next by $\int_{\check{\Omega}_{x}}[u(x, z)- \bar{u}(x)]dz= 0$,  we obtain
\begin{align}
\int_{T\check{\Omega}_{x_1}} [u(x, z)- \bar{u}(x)]dz= -\int_{\check{\Omega}_{x}- T\check{\Omega}_{x_1}}[u(x, z)- \bar{u}(x)]dz. \nonumber 
\end{align}

Therefore from Proposition \ref{prop directional deri est-any dim}, we have
\begin{align}
&\lim_{x_1\rightarrow x^+} |(II)|= \lim_{x_1\rightarrow x^+} |\frac{\frac{\mathrm{Jac}(T^{-1})}{h(x_1)}\int_{T\check{\Omega}_{x_1}} [u(x, z)- \bar{u}(x)]dz}{x_1- x} |\nonumber \\
&= \lim_{x_1\rightarrow x^+} |\frac{-(\frac{\mathrm{Jac}(T^{-1})}{h(x_1)}) \int_{\check{\Omega}_x- T\check{\Omega}_{x_1}} [u(x, z)- \bar{u}(x)]dz}{x_1- x} |\nonumber \\
&\leq \sup_{y\in \check{\Omega}_x}|D_yu|\cdot \epsilon\cdot \frac{1}{h(x)} \lim_{x_1\rightarrow x^+} |\frac{ V(\check{\Omega}_x- T\check{\Omega}_{x_1})}{x_1- x} |  \nonumber \\
&\leq C\epsilon\frac{\|x\|}{h(x)} \lim_{x_1\rightarrow x^+} |\frac{ h(x)- \frac{x^{n- 1}}{x_1^{n- 1}}h(x_1)}{x_1- x} | \nonumber \\
&\leq C\epsilon\frac{\|x\|\cdot x^{n- 1}}{h(x)} |(\frac{h(x)}{x^{n- 1}})'| \nonumber \\
&\leq C\epsilon\frac{\|x\|\cdot x^{n- 1}}{h(x)} \Big\{\frac{h'}{x^{n- 1}}+ \frac{h}{x^n}\Big\} \leq C\epsilon. \label{est of II}
\end{align}

\textbf{Step (3)}. Using $\Delta u= -\mu_1 u$, we get 
\begin{align}
\mu_1\cdot \int_0^x h\bar{u}ds= -\int_{\Omega_x}u_x(x, y)dy. \nonumber 
\end{align}

By (\ref{deri of bar u firstly}) we get
\begin{align}
\eta(x)&= h\tilde{u}'(x)+ \mu_1\int_0^x h\tilde{u}dt \nonumber \\
&= h\cdot \Big\{\frac{1}{h}\int_{\check{\Omega}_{x}} (D_yu\cdot \frac{z- y_0}{x}) dz+ \lim_{x_1\rightarrow x^+}(II)\Big\}+ \mu_1\int_0^x h[\tilde{u}(t)- \bar{u}(t)]dt .\label{general eta formula}
\end{align}

In the rest argument, we assume $x\leq \tau_0$. If $x\geq \tau_0$, note $\int_0^d \tilde{u}hdt= \int_0^d \bar{u}h dt= 0$, we have
\begin{align}
\eta(x)= h\cdot \Big\{\frac{1}{h}\int_{\check{\Omega}_{x}} (D_yu\cdot \frac{z- y_0}{x}) dz+ \lim_{x_1\rightarrow x^+}(II)\Big\}+ \mu_1\int_x^d h[\bar{u}(t)- \tilde{u}(t)]dt .\label{general eta formula-the right hand}
\end{align}
Similar argument as $x\leq \tau_0$ applies on (\ref{general eta formula-the right hand}).

Recall $c_1$ is defined in (\ref{def of tilde u}). Note that $\int_{0}^d h\bar{u}=0$, from Lemma \ref{lem grad est without k}, we have
\begin{align}
c_1=\frac{|\int_{0}^{\epsilon}h(x)(\bar{u}(x)-\bar{u}(\epsilon))dx+\int_{d-\epsilon}^{d}h(x)(\bar{u}(x)-\bar{u}(d-\epsilon))dx|}{\int_{0}^d h}\le C\epsilon^3. \label{est of c1}
\end{align}

Now by (\ref{est of c1}) and Lemma \ref{lem grad est without k}, note $x\in [\epsilon, d- \epsilon]$, we get 
\begin{align}
&\quad |\mu_1\int_0^x h[\tilde{u}(t)- \bar{u}(t)]dt| \leq C\cdot \int_0^x h|\hat{u}- \bar{u}|+ C\int_0^x h\cdot c_1 \nonumber \\
&\leq C\epsilon^2\int_{[0, \epsilon]} h+ 
C\epsilon^3\int_0^x h\leq Ch(x)\cdot \epsilon^3. \label{extra term in eta}
\end{align}

From (\ref{est of II}) and Proposition \ref{prop directional deri est-any dim}, for $x\in [\epsilon, d- \epsilon]$ we get
\begin{align}
\Big|h\cdot \Big\{\frac{1}{h}\int_{\check{\Omega}_{x}} (D_yu\cdot \frac{z- y_0}{x}) dz+ \lim_{x_1\rightarrow x^+}(II)\Big\}\Big|
&\leq Ch\epsilon.\nonumber 
\end{align}

By (\ref{general eta formula}) and the above we obtain the conclusion.
}
\qed

\section{The difference between eigenfunctions}\label{sec comp of eigen and cross}

In the rest of the paper,  unless otherwise mentioned,  the function $\zeta$ is the first Neumann eigenfunction of $N= ([0, d],  hdx)$ with respect to the first Neumann eigenvalue $\mu_1(N)$ and $\zeta(0)= \tilde{u}(0)$.

\begin{lemma}\label{lem grad est of ODE-zeta}
	{There is a universal constant $C= C(n)> 0$,  assume $\zeta(x_0)= 0$,  then 
		\begin{align}
			&\zeta'(x)\geq C^{-1}\cdot \|x\|,  \quad \quad \quad \forall \|x\|\in [0,  10^{-n}],  \nonumber \\
			& x_0\in [\frac{1}{10^n},  d- \frac{1}{10^n}],  \label{bound of zero points-zeta} \\
			&\sup_{x\in [0, d]}|\zeta(x)|\leq C,  \label{uniform bound of zeta}\\
			&0\leq \zeta'(x)\leq C\cdot \|x\|, \quad \quad \forall x\in [0, d]. \label{upper bound of zeta' in x}\\
			& 	C^{-1}\cdot V(\Omega)\leq \int_0^d h\cdot (\zeta)^2\leq C\cdot V(\Omega),\label{integra bound of zeta} \\
			&C^{-1}\cdot V(\Omega)\leq \int_0^d h\cdot (\zeta')^2\leq C\cdot V(\Omega).\label{integra bound of zeta'}
		\end{align}
	}
\end{lemma}	

\pf
{It follows from Theorem \ref{thm existence of phi to ODE},  we get that $\zeta= -\tilde{u}(0)\cdot \phi$,  note $|\tilde{u}(0)|= k\geq 1- \epsilon$,  the conclusion follows from Lemma \ref{lem grad est of ODE} and Lemma \ref{lem L2 upper bound deri of phi}.  
}
\qed

\begin{lemma}\label{lem mu1 equ with bar u}
{We have the following formula:
\begin{align}
	\mu_1(\Omega)= \frac{\int_0^d h(\tilde{u}')^2}{\int_0^d h\tilde{u}^2}- \frac{\int_0^d \eta \tilde{u}'}{\int_0^d h\tilde{u}^2}\geq \mu_1(N)- \frac{\int_0^d \eta \tilde{u}'}{\int_0^d h\tilde{u}^2}. \label{mu1 equ}
\end{align}
}
\end{lemma}

\pf
{From (\ref{def of eta}),  we get 
\begin{align}
	\int_0^d h(\tilde{u}')^2= \int_0^d \eta \tilde{u}'- \int_0^d \Big\{\tilde{u}'(x)\cdot \int_0^x \mu_1(\Omega) h\tilde{u}\Big\}dx.  \nonumber 
\end{align}

Using $\int_0^d h\tilde{u}= 0$ and integration by parts,  from Definition \ref{def mu1 N} we have 
\begin{align}
\mu_1(\Omega)= \frac{\int_0^d h(\tilde{u}')^2}{\int_0^d h\tilde{u}^2}- \frac{\int_0^d \eta \tilde{u}'}{\int_0^d h\tilde{u}^2}\geq \mu_1(N)- \frac{\int_0^d \eta \tilde{u}'}{\int_0^d h\tilde{u}^2}. \nonumber 
\end{align}
}
\qed

\begin{prop}\label{prop C0-1 close between ODE and PDE solution}
{There is a universal constant $C> 0$, such that
\begin{align}
\sup_{x\in [0, d]}|\tilde{u}- \zeta|+ \sup_{x\in [\epsilon, d-\epsilon]} |\tilde{u}'- \zeta'|\leq C\cdot \epsilon.  \nonumber 
\end{align}
}
\end{prop}

\pf
{\textbf{Step (1)}. In this step, we assume $x\leq \epsilon$. 

Because $\zeta$ is increasing on $[0, d]$,  we get 
\begin{align}
\zeta(x)\geq \zeta(0). \label{bound of phi by phi 0}
\end{align}

Then for $x\leq \epsilon$, from Lemma \ref{lem grad est of ODE-zeta},
\begin{align}
&|\tilde{u}- \zeta|(x)= |\zeta(x)- \zeta(0)| \leq |\int_0^x\zeta'(t)dt|\leq C\int_0^x tdt\leq Cx^2\leq C\epsilon^2 .\label{initial est of diff near end points}
\end{align}

\textbf{Step (2)}. In this step, we assume $x\in [\epsilon, \min\{\tau_0, d-\epsilon\}]$. Define $\tilde{\eta}(x)= \eta(x)- \int_0^x (\mu_1(\Omega)- \mu_1(N))h\tilde{u}ds$ and $\xi(x)= \tilde{u}(x)- \zeta(x)$. Recall
\begin{align}
h\tilde{u}'= \eta- \int_0^x \mu_1(\Omega) h\tilde{u}, \quad \quad h\zeta'= -\int_0^x \mu_1(N) h\zeta ds. \nonumber 
\end{align}

Note $\tilde{u}'= 0$ on $[0, \epsilon]\cup [d- \epsilon, d]$. From Lemma \ref{lem rough est of eta for general dim} we have 
\begin{align}
|\int_0^d \eta \tilde{u}'|\leq C\max_{t\in [0, d]}h(t)\cdot \epsilon\int_\epsilon^{d- \epsilon} |\tilde{u}'|\leq CV(\Omega)\epsilon \cdot \max|\tilde{u}|\leq CV(\Omega)\epsilon.  \label{norminator first estimate-higher dim}
\end{align}

From Lemma \ref{lem grad est without k},  we have 
\begin{align}
\inf_{x\in [d-\epsilon,d]}u(x,y)\ge 1-4C_1^2\epsilon^2.  \nonumber
\end{align}

By Proposition \ref{prop directional deri est-any dim},  we know that 
\begin{align}
|Du(x,y)|\le 10C_1^2(d-x),  \quad \quad \quad \forall \frac{d}{2}\leq x\le d-\epsilon.  \nonumber 
\end{align}

Thus for $x\in [d-\frac{1}{100C_1^2},d]$, we have
\begin{align}
u(x)\ge 1-4C_1^2\epsilon^2-10C_1^2\frac{1}{100C_1^2}\ge \frac{1}{2}.  \nonumber 
\end{align}
Now we get
\begin{align}
\int_{\Omega} h\bar{u}^2\ge \int_{\Omega\cap \{(x,  y): x\ge d-\frac{1}{100C_1^2}\}}\frac{1}{4}\ge CV(\Omega). \nonumber 
\end{align}

Because $\max |u|= 1$ and $d\leq 2$, we can get $\int_0^d h\cdot (\bar{u})^2\leq CV(\Omega)$ directly. Therefore
\begin{align}
C^{-1}\cdot V(\Omega)\leq \int_0^d h\cdot \bar{u}^2\leq C\cdot V(\Omega). \label{integra bound of bar u}
	\end{align}

On the other hand, from Lemma \ref{lem grad est without k},  we can get
\begin{align}
&|\int_0^d h\tilde{u}^2- \int_0^d h\bar{u}^2|= |\int_{0}^\epsilon h\cdot [\bar{u}(\epsilon)^2- \bar{u}^2(t)]+ \int_{d- \epsilon}^{d}h\cdot [\bar{u}(d- \epsilon)^2- \bar{u}^2(t)]| \nonumber \\
&\leq CV(\Omega)\epsilon^2. \nonumber 
\end{align}
Combining (\ref{integra bound of bar u}), we get
\begin{align}
	&C^{-1}\cdot V(\Omega)\leq \int_0^d h\cdot (\tilde{u})^2\leq C\cdot V(\Omega). \label{L2 upper bound of bar u}
	\end{align}

Now by Lemma \ref{lem mu1 equ with bar u}, (\ref{L2 upper bound of bar u}) and (\ref{norminator first estimate-higher dim}),  there is a universal constant $C> 0$ such that
\begin{align}
\mu_1(\Omega)\geq \mu_1(N)- C\epsilon . \label{NW eigenvalue comparison}
\end{align}

\textbf{Step (3)}. From (\ref{NW eigenvalue comparison}) and Lemma \ref{lem rough est of eta for general dim}, we can obtain
\begin{align}
|\tilde{\eta}|(x)&\leq C\epsilon h . \nonumber 
\end{align}

Now we get
\begin{align}
h\xi'&=h(\tilde{u}- \zeta)'= \eta- \int_0^x (\mu_1(\Omega)- \mu_1(N))h\tilde{u}+ \int_0^x \mu_1(N) h(\zeta- \tilde{u}) \nonumber \\
&= \tilde{\eta}- \int_0^x \mu_1(N)h\cdot \xi. \label{hk' eqa}
\end{align}

Note $h$ is increasing in $[0, \tau_0]$. Let $\mathcal{L}(x)= \sup_{s\leq x}|\xi(s)|$, then 
\begin{align}
\mathcal{L}'(x)&\leq |\xi'(x)|\leq \frac{|\tilde{\eta}|+ \int_0^x \mu_1(N) h|\xi|ds}{h(x)}\leq C\epsilon + \int_0^x \mu_1(N) |\xi|ds \nonumber \\
&\leq C\epsilon + C\cdot \mathcal{L}(x). \label{deri of k upper bound}
\end{align}

The above inequality implies 
\begin{align}
(e^{-Ct}\mathcal{L}(t))'\leq C\epsilon e^{-Ct}\leq C\epsilon , \quad \quad \quad \forall t\in [\epsilon, \tau_0]. \nonumber 
\end{align}

Note $\mathcal{L}(\epsilon)\leq C\epsilon^2$, from the above we get 
\begin{align}
\mathcal{L}(x)\leq e^{Cx}C\epsilon\cdot \int_\epsilon^x 1dt+ \mathcal{L}(\epsilon)\leq C\cdot\epsilon, \quad \quad \forall x\in [\epsilon, \tau_0]. \label{upper bound of k-first interval}
\end{align}

Using (\ref{deri of k upper bound}), we have
\begin{align}
|\xi'|(x)\leq C\epsilon , \quad \quad \forall x\in [\epsilon, \tau_0]. \label{upper bound of k'-first interval}
\end{align}

\textbf{Step (4)}. Next we consider the case $x\in [\tau_0, d- \epsilon]$ and define 
\begin{align}
\mathcal{R}(x)\vcentcolon= \sup_{t\in [x, d- \epsilon]}|\xi(t)|.  \nonumber 
\end{align}

Similar to (\ref{hk' eqa}), we have
\begin{align}
&h\xi'(x)= \tilde{\eta}(x)+ \int_x^d \mu_1(N)h \xi dt,  \nonumber \\
&|\tilde{\eta}(x)|\leq C\epsilon h(x).  \nonumber 
\end{align}

Note $|\xi|\leq C$ by Lemma \ref{lem grad est of ODE-zeta} and $|\tilde{u}|\leq 1$.  Now for any $x\in [\tau_0, d- \epsilon]$, we get
\begin{align}
-\mathcal{R}'(x)&\leq |\xi'(x)|\leq \frac{|\tilde{\eta}|(x)+ \int_x^d \mu_1(N)h |\xi|dt}{h(x)}\leq C\epsilon + \mu_1(N)\int_x^d |\xi|dt \nonumber \\
&\leq C\epsilon + C\cdot \mathcal{R}(x)+ \int_{d- \epsilon}^dC\cdot |\xi| \label{later come back ineq} \\
&\leq C\epsilon + C\cdot \mathcal{R}(x). \label{deri of k on the right side} 
\end{align}

Now we get 
\begin{align}
-(e^{Ct}\mathcal{R}(t))'\leq Ce^{-Ct}\epsilon\leq C\epsilon,  \quad \quad \forall t\in [\tau_0, d- \epsilon].  \nonumber
\end{align}

Taking the integration from $x$ to $d- \epsilon$ in the above, we have
\begin{align}
\mathcal{R}(x)\leq Ce^{-Cx}\Big\{C\epsilon\int_{x}^{d- \epsilon} 1dt+ Ce^{C(d- \epsilon)}\mathcal{R}(d- \epsilon)\Big\} \leq C\cdot (\epsilon+ \mathcal{R}(d- \epsilon)).  \label{upper bound of right function}
\end{align}

From (\ref{deri of k on the right side}) and (\ref{upper bound of right function}), we obtain
\begin{align}
\xi(x)&= \xi(d- \epsilon)- \int_x^{d-\epsilon}\xi'(s)ds\geq \xi(d- \epsilon)- \int_x^{d- \epsilon} C\epsilon + C\mathcal{R}(s)ds \nonumber \\
&\geq \xi(d- \epsilon)- C\epsilon- C(d- x)\mathcal{R}(x)\nonumber \\
&\geq \xi(d- \epsilon)- C\epsilon- C(d- x)(\epsilon+ \mathcal{R}(d- \epsilon))\nonumber \\
&\geq [1- C_1(d- x)]\xi(d- \epsilon)- C\epsilon.  \nonumber 
\end{align}

Therefore we have
\begin{align}
\xi(d- \epsilon)\leq \frac{\xi(x)+ C\epsilon}{1- C_1(d-x)},  \quad \quad \forall x\in[\max\{\tau_0, d- \frac{1}{2C_1}\}, d-\epsilon].  \label{upper bound of k at d-eps}
\end{align}

\textbf{Step (5)}. Now choose $\delta\in [\epsilon, d- \epsilon]$, for $t\in [\epsilon, d- \delta]$, using the concave property of $h^{\frac{1}{n-1}}$, we have
\begin{align}
		\int_0^t \frac{h(s)}{h(t)}ds\leq \int_0^t (\frac{d- s}{d- t})^{n-1}ds\leq \frac{C}{(d- t)^{n-1}}\leq \frac{C}{\delta^{n-1}}.  \nonumber 
\end{align}

Now we compute $\mathcal{L}'$ again to obtain
\begin{align}
\mathcal{L}'(t)&\leq |\xi'(t)|\leq \frac{|\tilde{\eta}|+ \int_0^t \mu_1(N)h(s)|\xi|ds}{h(t)}\leq C\epsilon + \int_0^t C\frac{h(s)}{h(t)} |\xi|ds \nonumber \\
&\leq  C\epsilon + \frac{C}{\delta^{n-1}}\mathcal{L}(t). \nonumber 
\end{align}

Using (\ref{initial est of diff near end points}), we know that $\mathcal{L}(\epsilon)\leq C\epsilon^2$.  Integrating the above inequality from $\epsilon$ to $t$, yields
\begin{align}
\mathcal{L}(t)\leq e^{\frac{C}{\delta^{n-1}}(t- \epsilon)}\cdot C\epsilon,  \quad \quad \forall t\in [\epsilon, d- \delta].  \label{Lt general bound}
\end{align}

Now we choose $\delta= \frac{1}{2C_1}$ where $C_1$ is from (\ref{upper bound of k at d-eps}), then by (\ref{Lt general bound}) we get
\begin{align}
\xi(d- \frac{1}{2C_1})\leq C\epsilon.  \label{upper bound of k near d}
\end{align}

If $\tau_0\geq d- \frac{1}{2C_1}$, note $\xi(\tau_0)\leq C\epsilon$, choose $x= \tau_0$ in (\ref{upper bound of k at d-eps}) we get
\begin{align}
\xi(d- \epsilon)\leq \frac{\xi(\tau_0)+ C\epsilon}{1- C_1(d- \tau_0)}\leq C\epsilon.  \nonumber  
\end{align}
If $\tau_0< d- \frac{1}{2C_1}$, choosing $x= d- \frac{1}{2C_1}$ in (\ref{upper bound of k at d-eps}), using (\ref{upper bound of k near d}), we also get
\begin{align}
\xi(d- \epsilon)\leq C\epsilon. \label{xi at d-eps 1}
\end{align}

By (\ref{xi at d-eps 1}) and (\ref{upper bound of right function}),  we get
\begin{align}
\sup_{x\in [\tau_0, d-\epsilon]}|\xi(x)|\leq \mathcal{R}(x)\leq C\cdot \epsilon.  \label{sup diff in the mid interval}
\end{align}

\textbf{Step (6)}. For $x\in [d- \epsilon, d]$,  using Lemma \ref{lem grad est of ODE-zeta} and (\ref{sup diff in the mid interval}),  we have
\begin{align}
|\xi(x)|&\leq |\xi(x)- \xi(d- \epsilon)|+ |\xi(d- \epsilon)| \leq  |\zeta(x)- \zeta(d- \epsilon)|+ C\epsilon \nonumber \\
&\leq \sup_{t\in [d- \epsilon, d]}|\zeta'(t)|\epsilon+ C\epsilon \leq C\epsilon.  \label{small near the right end point} 
\end{align}

Plugging (\ref{small near the right end point}) into (\ref{later come back ineq}),  for $x\in [\tau_0, d- \epsilon]$ we get 
\begin{align}
-\mathcal{R}'(x)&\leq |\xi'(x)|\leq C\epsilon + C\cdot \mathcal{R}(x)+ \int_{d- \epsilon}^dC\cdot |\xi|\leq C\epsilon+ C\cdot \mathcal{R}(x); \label{deri of k on the right side-1} \\
\mathcal{R}(x)&\leq C(\epsilon^2+ \mathcal{R}(d- \epsilon)). \label{upper bound of right function-1}
\end{align}

Note for $x\in [\epsilon, d- \epsilon]$,  we have $|h'(x)|\leq 1$.  From (\ref{upper bound of k'-first interval}), (\ref{deri of k on the right side-1}),  we get
\begin{align}
\sup_{x\in [\epsilon, d-\epsilon]}|\tilde{u}'- \zeta'|\leq C\epsilon.  \nonumber 
\end{align}
}
\qed

\section{The integral estimate of vertical directional derivatives}\label{sec dire derivative est}

\begin{lemma}\label{lem integ est of Dy u}
{There is a universal constant $C> 0$ such that
\begin{align}
\int_\Omega|D_y u|^2\leq C\epsilon\cdot V(\Omega). \nonumber 
\end{align}
}
\end{lemma}

\pf
{\textbf{Step (1)}. Note $\zeta$ is defined on $[0, d]$, the integral $\int_\Omega \zeta^2$ is defined as follows:
\begin{align}
\int_\Omega \zeta^2\vcentcolon= \int_0^d dx \int_{\Omega_x}\zeta^2(x)dy= \int_0^d h(x)\zeta^2(x)d, \nonumber 
\end{align}
where $y\in \mathbb{R}^{n- 1}$. Similar definition applies on $\bar{u}, \tilde{u}$.

Note we have
\begin{align}
\int_\Omega |u_y|^2&= \int_\Omega |Du|^2- |u_x|^2= \mu_1(\Omega)
\int_\Omega |u|^2- \int_\Omega |u_x|^2 \nonumber \\
&\leq \mu_1(N)\int_\Omega u^2- \int_\Omega |u_x|^2= (I)+ (II);\nonumber 
\end{align}
where 
\begin{align}
 (I)= \mu_1(N)\int_\Omega [u^2- \zeta^2], \quad \quad  (II)= \int_\Omega [(\zeta')^2- u_x^2]
\end{align}

From Lemma \ref{lem grad est without k} and $\int_0^d h\bar{u}= 0$, we have
\begin{align}
&\quad |\bar{u}- \tilde{u}|(x)\leq |\bar{u}- \hat{u}|(x)+ |\hat{u}- \tilde{u}|(x) \nonumber \\
&\leq \sup_{x_1\in [0, \epsilon]}|\bar{u}(x_1)- \bar{u}(\epsilon)|+ \sup_{x_2\in [d- \epsilon, d]}|\bar{u}(x_2)- \bar{u}(d- \epsilon)|+  |\frac{\int_0^d h\hat{u}}{\int_0^d h}| \nonumber \\
&\leq C\cdot \epsilon^2+ |\frac{\int_0^d h\bar{u}- h\hat{u}}{\int_0^d h}| \leq C\epsilon^2. \label{diff between bar u and tilde u}
\end{align}

Then using Proposition \ref{prop C0-1 close between ODE and PDE solution} and (\ref{diff between bar u and tilde u}),  we have
\begin{align}
(I)&\leq C|\int_\Omega u^2- \bar{u}^2|+ C|\int_\Omega \bar{u}^2- \tilde{u}^2|+ C|\int_\Omega \tilde{u}^2- \zeta^2| \nonumber \\
&\leq C\cdot \int_\Omega (u- \bar{u})^2+ C\cdot \epsilon V(\Omega) \nonumber  .
\end{align}

On the other hand, note $\Omega_x$ is a convex domain with diameter $\leq\epsilon$,  then we get $\mu_1(\Omega_x)\geq \frac{C}{\epsilon^2}$ by Payne-Weinberger inequality (see \cite{PW}).  Hence
\begin{align}
\int_\Omega (u- \bar{u})^2&\leq \int_0^d dx\int_{\Omega_x} |u- \bar{u}|^2 dy \leq \Big(\max_{x\in [0, d]}\frac{1}{\mu_1(\Omega_x)}\Big)\cdot \int_\Omega |D_yu|^2\leq C\epsilon^2 \int_\Omega |Du|^2 \nonumber \\
&\leq C\epsilon^2\cdot V(\Omega).  \nonumber 
\end{align}

By the above we get
\begin{align}
(I)\leq C\epsilon\cdot V(\Omega).  \nonumber 
\end{align}

\textbf{Step (2)}. By Lemma \ref{lem grad est without k} and Lemma \ref{lem grad est of ODE-zeta}, we estimate $(II)$ as follows:
\begin{align}
(II)&\leq \int_{[0, \epsilon]\cup [d- \epsilon, d]} dx \int_{\Omega_x}|u_x^2|+ |\zeta'^2|+ \int_{[\epsilon, d- \epsilon]} dx \int_{\Omega_x}|u_x^2- \zeta'^2|\nonumber \\
&= C\epsilon^2 V(\Omega)+ (II)_1+ (II)_2; \nonumber 
\end{align}
where
\begin{align}
(II)_1= \int_{[\epsilon, d- \epsilon]} dx \int_{\Omega_x}|u_x^2- \tilde{u}'^2|, \quad\quad \quad (II)_2= \int_{[\epsilon, d- \epsilon]} dx \int_{\Omega_x} |\tilde{u}'^2- \zeta'^2|  \nonumber 
\end{align}

By Proposition \ref{prop C0-1 close between ODE and PDE solution},  we get
\begin{align}
(II)_2\leq C\epsilon\cdot V(\Omega).  \nonumber 
\end{align}

By Proposition \ref{prop C0-1 close between ODE and PDE solution}, Lemma \ref{lem grad est of ODE-zeta} and $|Du|\leq C$, we get
\begin{align}
\sup_{t\in [\epsilon, d- \epsilon]}|u_x|+ |\tilde{u}'|\leq C. \label{bar u' ineq}
\end{align}

By the convexity property of $\Omega$ and $\frac{\partial u}{\partial \vec{n}}\big|_{\partial\Omega}= 0$, we have $\displaystyle \Big(D_{\vec{n}}D_{\vec{t}}w\cdot D_{\vec{t}}w\Big)\big|_{\partial \Omega}\leq 0$, where $\vec{t}$ is the unit tangent vector of $\partial \Omega$. 

Then
\begin{align}
\frac{1}{2}\int_{\Omega}\Delta(|Du|^2)= \frac{1}{2}\int_{\partial\Omega}\frac{\partial}{\partial\vec{n}}|Du|^2= \int_{\partial\Omega}D_{\vec{n}}D_{\vec{t}}u\cdot D_{\vec{t}}u\leq 0.  \nonumber 
\end{align}

On the other hand, 
\begin{align}
\frac{1}{2}\int_{\Omega}\Delta(|Du|^2)= \int_{\Omega}|D^2u|^2+ D(\Delta u)\cdot Du= \int_{\Omega}|D^2u|^2- |\Delta u|^2.  \nonumber 
\end{align}

We have 
\begin{align}
\int_\Omega |D^2 u|^2\leq \int_\Omega |\Delta u|^2.  \label{Hessian est by Lap on convex for Neumann func}
\end{align}

Using (\ref{bar u' ineq}), (\ref{Hessian est by Lap on convex for Neumann func}) and $\mu_1(\Omega_x)\geq \frac{C}{\epsilon^2}$, using Lemma \ref{lem rough est of eta for general dim} we have 
\begin{align}
(II)_1&\leq C\int_\epsilon^{d- \epsilon} dx\int_{\Omega_x} \Big|u_x- \tilde{u}'|dy\nonumber \\
&\leq C\int_\epsilon^{d- \epsilon} dx\int_{\Omega_x} |u_x- \frac{1}{h}\int_{\Omega_x}u_x|dy+ C\int_\epsilon^{d- \epsilon} dx\int_{\Omega_x}|\frac{1}{h}\int_{\Omega_x}u_x- \tilde{u}'| dy  \nonumber \\
&\leq C\int_\epsilon^{d- \epsilon} \sqrt{h}\cdot \Big\{\int_{\Omega_x} \Big|u_x- \frac{1}{h}\int_{\Omega_x}u_x dy \Big|^2\Big\}^{\frac{1}{2}} dx+ C\epsilon V(\Omega)\nonumber \\
&\leq C\int_\epsilon^{d- \epsilon} \sqrt{h}\cdot \Big\{\frac{1}{\mu_1(\Omega_x)} \int_{\Omega_x} |D^2 u|^2\Big\}^{\frac{1}{2}} dx+ C\epsilon V(\Omega)\nonumber \\
&\leq C\epsilon\cdot (\int_\epsilon^{d- \epsilon} h)^\frac{1}{2}\cdot \Big\{\int_\epsilon^{d- \epsilon}dx\int_{\Omega_x} |D^2 u|^2\Big\}^{\frac{1}{2}}+ C\epsilon V(\Omega)  \nonumber \\
&\leq C\epsilon\cdot V(\Omega)^{\frac{1}{2}}\cdot (\int_\Omega |\Delta u|^2)^{\frac{1}{2}}+ C\epsilon V(\Omega) \nonumber \\
&\leq C\epsilon\cdot V(\Omega).  \nonumber 
\end{align}

The conclusion follows from the above.  
}
\qed

\part{The planar convex domain}

\textbf{In the rest of the paper, we assume $\Omega\subseteq \mathbb{R}^2$}. 

\section{Key equations for planar convex domain}\label{sec key equality}

\begin{lemma}\label{lem pj-width is equal to width}
{We have
\begin{align}
&\mathrm{Width}(\Omega)= \mathscr{PW}(\Omega)= \epsilon= \hat{\epsilon}, \quad \quad t_0= \tau_0,  \nonumber \\
&\Omega= \{(x, y)\in \mathbb{R}^2: x\in [0, d], h_-(x)\leq y\leq h_+(x)\}, \nonumber \\
&h(x)= h_+(x)- h_-(x), \quad \quad   |h'(x)|= |h_+'(x)|+ |h_-'(x)|,  \quad \forall x\in [0, d]. \label{deri of h control h- and h+ deri}
\end{align}
}
\end{lemma}

\begin{remark}\label{rem key equ for 2-dim case}
{For the equation (\ref{deri of h control h- and h+ deri}), also see \cite[Lemma $3.1$]{GJ}.
}
\end{remark}

\pf
{Firstly we have
\begin{align}
\Omega= \{(x,y): x\in [0, d],y\in [h_-(x), h_+(x)]\}. \nonumber 
\end{align}

Let $\vec{n}=\frac{(-h_{+}'(\tau_0),1)}{\sqrt{h_{+}'(\tau_0)^2+1}}$.  By  $\Omega$ is convex,  we know that $\Omega$ lies on one connected component of $\mathbb{R}^2- \gamma$,  where $\gamma$ is the line passing $(\tau_0, h_+(\tau_0))$ with $\gamma'\cdot \vec{n}= 0$.

Hence we have $\Omega\subset \{(x,y)\in \R^2|(x,y)\cdot \vec{n}\le  (\tau_0,h_{+}(\tau_0))\cdot \vec{n}\}$, so 
\begin{align}
\sup_{(x,y)\in \Omega}(x,y)\cdot \vec{n}= (\tau_0,h_{+}(\tau_0))\cdot \vec{n}.  \nonumber 
\end{align}

Note $h'(\tau_0)= 0$, we get $h_+'(\tau_0)= h_-'(\tau_0)$. Similarly we have 
\begin{align}
\inf_{(x,y)\in \Omega}(x,y)\cdot \vec{n}= (\tau_0,h_{-}(\tau_0))\cdot \vec{n}.  \nonumber 
\end{align}

By definition of $\mathrm{Width}(\Omega)$, we have 
\begin{align}
\epsilon&=\mathrm{Width}(\Omega)\leq \sup_{p, q\in \Omega}(p- q)\cdot \vec{n} \nonumber \\
&= \Big[(\tau_0,h_{+}(\tau_0))- (\tau_0,h_{-}(\tau_0))\Big]\cdot \vec{n}=\frac{h(\tau_0)}{\sqrt{h_{+}'(\tau_0)^2+1}}= \frac{\epsilon}{\sqrt{h_{+}'(\tau_0)^2+1}}.  \nonumber 
\end{align}

Hence $h_{+}'(\tau_0)=0$.  Similarly we get $h_-'(\tau_0)= 0$, and we have
\begin{align}
|h'(x)|= |h_+'(x)|+ |h_-'(x)|, \quad \quad \forall x\in [0, d]. \nonumber 
\end{align}
}
\qed

\section{The improved estimates of error terms}\label{sec improved est}

\begin{lemma}\label{lem improved est of eta-2dim}
{There is a universal constant $C> 0$, such that 
\begin{align}
|\eta|(x)\leq |h'(x)| \cdot \sqrt{h(x)}\cdot \sqrt{\int_{\Omega_x} |D_yu|^2} + Ch(x)\epsilon^3, \quad \quad \forall x\in [\epsilon, d- \epsilon].\nonumber 
\end{align}
}
\end{lemma}

\pf
{Now we can write $\int_{\Omega_x}u dy$ as follows:
\begin{align}
\int_{\Omega_x}u dy= \frac{1}{h}\int_{h_-(x)}^{h_+(x)} u(x,y)dy, \nonumber 
\end{align}

Hence we have
\begin{align}
&\bar{u}'- \frac{1}{h}\int_{\Omega_x}u_xdy+\frac{h'}{h^2}\int_{\Omega_x}udy \nonumber \\
&= \frac{1}{h}\Big\{u(x, h_+(x ) )\cdot (h_+)_x(x )- u(x, h_-(x ) )\cdot (h_-)_x(x )\Big\}. \label{bar u deriv} 
\end{align}

Define 
\begin{align}
\hat{\eta}\vcentcolon= h\cdot \bar{u}'- \int_{\Omega_x}u_xdy. \label{def of hat eta} 
\end{align}
Assume $(x, \psi(x))\in \Omega_x$ such that 
\begin{align}
\frac{1}{h}\int_{\Omega_x}udy = u(x, \psi(x)). \label{choice of psi x in ave integral} 
\end{align}

From (\ref{bar u deriv}), (\ref{def of hat eta}), (\ref{choice of psi x in ave integral}) and  (\ref{deri of h control h- and h+ deri}), we have
\begin{align}
&\quad |\hat{\eta}| \nonumber \\
&\leq |\Big\{u(x, h_+(x ) )\cdot (h_+)_x(x )- u(x, h_-(x ) )\cdot (h_-)_x(x )\Big\} - \frac{h'}{h}\int_{\Omega_x}udy|  \nonumber \\
&= |u(x, h_+(x ) )\cdot (h_+)_x(x )- u(x, h_-(x ) )\cdot (h_-)_x(x )- h'(x)\cdot u(x, \psi(x))| \nonumber \\
&= |\big[u(x, h_+(x ) )- u(x, \psi(x))\big]\cdot (h_+)_x(x ) - \big[u(x, h_-(x ) )- u(x, \psi(x))\big]\cdot (h_-)_x(x )| \nonumber\\
&\leq \big[\int_0^1 |D_{y}u\big(x, t(h_+ )+ (1- t)\psi(x)\big)|\cdot |(h_+ )- \psi(x)|]\cdot |(h_+)_x(x )|dt \nonumber \\
&\quad + \big[\int_0^1 |D_{y}u\big(x, t(h_- )+ (1- t)\psi(x)\big)|\cdot |(h_- )- \psi(x)|\big]\cdot |(h_-)_x(x )|dt \nonumber \\
&\leq \int_{\Omega_x} |D_yu|(x, y)\cdot \big\{|h_+)_x(x )|+ |(h_-)_x(x )|\big\}dy \nonumber \\
&\leq \big\{|h_+)_x(x )|+ |(h_-)_x(x )|\big\}\cdot \sqrt{h(x)}\cdot \sqrt{\int_{\Omega_x} |D_yu|^2} \nonumber \\
&= |h'(x)| \sqrt{h(x)}\cdot \sqrt{\int_{\Omega_x} |D_yu|^2}\label{est of hat eta}
\end{align}

By (\ref{extra term in eta}), (\ref{general eta formula}) and (\ref{est of hat eta}), we get
\begin{align}
|\eta|(x)&= |h\tilde{u}'(x)+ \mu_1\int_0^x h\tilde{u}dt| \nonumber \\
&\leq |h\bar{u}'(x)+ \mu_1\int_0^x h\bar{u}dt|+ \mu_1|\int_0^x h[\tilde{u}- \bar{u}]dt| \nonumber \\
&\leq |\hat{\eta}|+ Ch(x)\epsilon^3\nonumber \\
&\leq |h'(x)| \cdot \sqrt{h(x)}\cdot \sqrt{\int_{\Omega_x} |D_yu|^2} + Ch(x)\epsilon^3. \nonumber 
\end{align}
}
\qed

Using Lemma \ref{lem grad est of ODE-zeta}, there is $C> 0$ such that
\begin{align}
|\zeta'|\leq C\|x\|, \quad \quad \forall x\in [0, d]. \label{pt est of tilde u'}
\end{align}

\begin{prop}\label{prop improved est of eigenvalue gap}
{There is a universal constant $C> 0$ such that
\begin{align}
\int_0^d |\eta\tilde{u}'|dx\leq C\epsilon^{\frac{3}{2}}\sqrt{V(\Omega)}\cdot \sqrt{\int_0^d \frac{|h'|^2\|x\|^2}{h}dx}+ CV(\Omega)\cdot \epsilon^3. \nonumber 
\end{align}
}
\end{prop}

\pf
{Using Proposition \ref{prop C0-1 close between ODE and PDE solution}, we get
\begin{align}
|\tilde{u}'|\leq |\zeta'|+ C\epsilon. \nonumber 
\end{align}

From Lemma \ref{lem improved est of eta-2dim}, we have
\begin{align}
&\int_0^d |\eta\tilde{u}'|dx\leq \int_{\epsilon}^{d- \epsilon} |\eta|\cdot (|\zeta'|+ C\epsilon) \nonumber \\
&\leq \int_{\epsilon}^{d- \epsilon} |\zeta'|\cdot |h'|\cdot \sqrt{h}\cdot \sqrt{\int_{\Omega_x}|D_yu|^2} dx+ C\epsilon^3\int_0^d h|\zeta'| \nonumber \\
&+ C\epsilon\int_{\epsilon}^{d- \epsilon}  |h'|\cdot \sqrt{h}\cdot \sqrt{\int_{\Omega_x}|D_yu|^2} dx
+ C\epsilon^4\int_0^d hdx. \label{first ineq of integral of eta bar u}
\end{align}

Using (\ref{first ineq of integral of eta bar u}), (\ref{pt est of tilde u'}) and Lemma \ref{lem integ est of Dy u}, we have
\begin{align}
&\int_0^d |\eta\tilde{u}'|dx\leq C\int_{\epsilon}^{d- \epsilon} \|x\|\cdot |h'|\cdot \sqrt{h}\cdot \sqrt{\int_{\Omega_x}|D_yu|^2} dx+ CV(\Omega)\cdot \epsilon^3 \nonumber \\
&\leq C\cdot \sqrt{\int_0^d h^2(\int_{\Omega_x}|D_yu|^2dy) dx}\cdot \sqrt{\int_0^d \frac{|h'|^2\|x\|^2}{h}dx}+ CV(\Omega)\cdot \epsilon^3 \nonumber \\
&\leq C\epsilon\cdot \sqrt{\int_0^d \int_{\Omega_x}|D_yu|^2dy dx}\cdot \sqrt{\int_0^d \frac{|h'|^2\|x\|^2}{h}dx}+ CV(\Omega)\cdot \epsilon^3 \nonumber\\
&\leq C\epsilon^{\frac{3}{2}}\sqrt{V(\Omega)}\cdot \sqrt{\int_0^d \frac{|h'|^2\|x\|^2}{h}dx}+ CV(\Omega)\cdot \epsilon^3. \nonumber
\end{align}
}
\qed

\section{The inequality of eigenvalue and width}\label{sec inequality}

\begin{theorem}\label{thm HW-conj confirmed}
{There are universal constants $C_1, C_2> 0$, such that for any convex domain $\Omega\subseteq \mathbb{R}^2$ with $\mathrm{diam}(\Omega)= 2$ and $\mathrm{Width}(\Omega)\in (0, C_1)$; we have 
\begin{align}
\mu_1(\Omega)\geq \frac{\pi^2}{4}+ C_2\cdot \mathrm{Width}(\Omega)^2.  \nonumber 
\end{align}
}
\end{theorem}

\begin{remark}\label{rem Liouville transform}
The transform $w= \sqrt{h} \zeta'$ is called Liouville transformation \cite{Liouville}.
\end{remark}

\pf
{\textbf{Step (1)}. Assume $10^{-2}<\tau_0<d-10^{-2}$, note that h is concave, then
\begin{align}
&\int_{10^{-2}}^{\tau_0}\frac{(h')^2}{h}\cdot \|x\|^2dx\le \frac{C}{\ep}\int_{10^{-2}}^{\tau_0}(h')^2\le \frac{C}{\ep}(h'(10^{-2}))^2 \nonumber\\
&\int_{0}^{10^{-2}}\frac{(h')^2}{h}\cdot \|x\|^2dx\ge \int_{10^{-2}/2}^{10^{-2}}\frac{(h')^2}{h}\cdot \|x\|^2dx\ge \frac{C}{\ep}\int_{10^{-2}/2}^{10^{-2}}(h')^2\ge \frac{C}{\ep}(h'(10^{-2}))^2 \nonumber
\end{align}
So we get $\int_{10^{-2}}^{\tau_0}\frac{(h')^2}{h}\cdot \|x\|^2dx\le 
C\int_{0}^{10^{-2}}\frac{(h')^2}{h}\cdot \|x\|^2dx$. 

Similarily we have $\int_{\tau_0}^{d-10^{-2}}\frac{(h')^2}{h}\cdot \|x\|^2dx\le C\int_{d-10^{-2}}^{d}\frac{(h')^2}{h}\cdot \|x\|^2dx$. Hence we get 
\begin{align}
\int_{[0, 10^{-2}]\cup [d- 10^{-2}, d]}\frac{(h')^2}{h}\cdot \|x\|^2dx\geq C\cdot \int_0^d \frac{(h')^2}{h}\cdot \|x\|^2dx. \label{bound of 2nd error local}
\end{align}

If $\tau_0>d-10^{-2}$,then
\begin{align}
		&\int_{10^{-2}}^{d-10^{-2}}\frac{(h')^2}{h}\cdot \|x\|^2dx\le \frac{C}{\ep}\int_{10^{-2}}^{d-10^{-2}}(h')^2\le \frac{C}{\ep}(h'(10^{-2}))^2\le C \int_{0}^{10^{-2}}\frac{(h')^2}{h}\cdot \|x\|^2dx.\nonumber
\end{align}
If $\tau_0<10^{-2}$, then $\int_{10^{-2}}^{d-10^{-2}}\frac{(h')^2}{h}\cdot \|x\|^2dx\le C\int_{d-10^{-2}}^{d}\frac{(h')^2}{h}\cdot \|x\|^2dx$.

Let $\mathcal{A}= [0, 10^{-2}]\cup [d- 10^{-2}, d]$ in the rest argument. From (\ref{bound of 2nd error local}), Lemma \ref{lem mu1 equ with bar u} and Proposition \ref{prop improved est of eigenvalue gap}, we get
\begin{align}
\mu_1(\Omega)&\ge \mu_1(N)- \frac{\int_0^d |\eta\tilde{u}'|dx}{\int_0^d h\tilde{u}^2} \nonumber \\
&\geq \mu_1(N)- C\frac{\epsilon^{\frac{3}{2}}\Big\{V(\Omega)\int_0^d  \frac{|h'|^2}{h}\|x\|^2 dx\Big\}^{\frac{1}{2}}}{\int_0^d h\zeta^2}- C\epsilon^3  \nonumber \\
&\geq \mu_1(N)- C\frac{\epsilon^{\frac{3}{2}}\Big\{V(\Omega)\int_{\mathcal{A}}  \frac{|h'|^2}{h}\|x\|^2 dx\Big\}^{\frac{1}{2}}}{\int_0^d h\zeta^2}- C\epsilon^3 . \label{eigenvalues ineq}
\end{align}

\textbf{Step (2)}. From Lemma \ref{lem grad est of ODE-zeta}, there is a universal constant $C>0$ such that
\begin{align}
|\zeta'|(x)\geq C\|x\|,   \quad \quad \forall x\in [0, 10^{-2}]\cup [d- 10^{-2}, d].\label{zeta' lower bound}
\end{align}

Let $w= \sqrt{h} \zeta'$, from $-(h\zeta')'= \mu_1(N)\cdot h\zeta$, we have
\begin{align}
		w''= -\mu_1(N) w+ (\frac{3}{4}(\frac{h'}{h})^2- \frac{h''}{2h})w, \quad \quad and \quad \quad w(0)= w(d)= 0. \nonumber
\end{align}

Also we get
\begin{align}
	\mu_1(N)= \frac{\int_0^d (w'(t))^2+ V\cdot w^2dt}{\int_0^d w(t)^2dt},  \quad \quad  V\vcentcolon= (\frac{3}{4}(\frac{h'}{h})^2- \frac{h''}{2h}). \label{expression of 1st Neug eigen to ODE} 
\end{align}

By $h''\leq 0$, from (\ref{eigenvalues ineq}), (\ref{zeta' lower bound}) and (\ref{expression of 1st Neug eigen to ODE}), using Cauchy-Schwartz inequality we have 
\begin{align}
&	\mu_1(\Omega)\geq \frac{\int_0^d (w')^2+ \frac{3}{4}\frac{(h')^2}{h^2}\cdot w^2 dt}{\int_0^d w^2dt}- C\frac{\epsilon^{\frac{3}{2}}\Big\{V(\Omega)\int_{\mathcal{A}}  \frac{|h'|^2}{h}\|x\|^2 dx\Big\}^{\frac{1}{2}}}{\int_0^d h\zeta^2}- C\epsilon^3\nonumber \\
	&\geq \frac{\pi^2}{d^2}+ \frac{\int_0^d \frac{3}{4}\frac{(h')^2}{h}(\zeta')^2dx- C\cdot \epsilon^{\frac{3}{2}}\Big\{V(\Omega)\int_{\mathcal{A}} \frac{|h'|^2}{h}\|x\|^2 dx\Big\}^{\frac{1}{2}}}{\int_0^d h(\zeta')^2}- C\epsilon^3\nonumber \\
	&\geq \frac{\pi^2}{d^2}+ \frac{C_1\int_{\mathcal{A}} \frac{(h')^2}{h}\|x\|^2dx- C\cdot \epsilon^{\frac{3}{2}}\Big\{V(\Omega)\int_{\mathcal{A}} \frac{|h'|^2}{h}\|x\|^2 dx\Big\}^{\frac{1}{2}}}{\int_0^d h(\zeta')^2}- C\epsilon^3\nonumber \\
&	\geq \frac{\pi^2}{d^2}+ \frac{\frac{C_1}{2}\int_{\mathcal{A}} \frac{(h')^2}{h}\|x\|^2dx- C_2\cdot \epsilon^{2+ \frac{1}{2}}V(\Omega) }{\int_0^d h(\zeta')^2}- C\epsilon^3  \nonumber \\
&\geq \frac{\pi^2}{d^2}+ C\frac{\int_{\mathcal{A}}  \frac{|h'|^2}{h}\|x\|^2 dx}{V(\Omega)} - C\epsilon^{2+ \frac{1}{2}} .  \label{NW-ineq with d and epsilon-general}
\end{align}

\textbf{Step (3)}. If $d\leq 2- 10^{-6n}\cdot \epsilon^2$, then the conclusion follows from (\ref{NW-ineq with d and epsilon-general}). In the rest argument, we will assume $d\geq 2- 10^{-6n}\cdot \epsilon^2$.

Similar to (\ref{bound of 2nd error local}), we have
\begin{align}
\int_{[0, 10^{-2}]} \frac{(h')^2}{h}\|x\|^2dx \geq C\cdot \int_{[0, 10^{-1}]} \frac{(h')^2}{h}\|x\|^2dx . \label{10-1 and 10-2}
\end{align}
Therefore from Lemma \ref{lem major 2nd term} and (\ref{10-1 and 10-2}),  we get 
\begin{align}
\int_{\mathcal{A}} \frac{(h')^2}{h}\|x\|^2dx\geq C\cdot \int_{[0, 10^{-1}]} \frac{(h')^2}{h}\|x\|^2dx\geq CV(\Omega)\epsilon^2. \label{first phi crucial est-slice}
\end{align}

Putting (\ref{first phi crucial est-slice}) into (\ref{NW-ineq with d and epsilon-general}), we get
\begin{align}
\mu_1(\Omega)\geq \frac{\pi^2}{4}+\frac{CV(\Omega)\epsilon^2}{V(\Omega)} - C\epsilon^{2+ \frac{1}{2}}\geq \frac{\pi^2}{4}+ C\epsilon^2. \nonumber 
\end{align}
}
\qed

\appendix

\section{The eigenvalue and Rayleigh quotient}

\begin{lemma}\label{lem existence of minimizer-new}
{There exists $\varphi\in H-\{0\}$ such that $\mathscr{R}(\varphi)= \min\limits_{f\in H-\{0\}}\mathscr{R}(f)$. Furthermore, $\varphi$ satisfies 
\begin{align}
			\int_0^d \varphi'\cdot v'hdt= \mathscr{R}(\varphi)\int_0^d \varphi\cdot vhdt ,  \quad \quad \forall v\in W^{1,2}(0,d).  \nonumber 
\end{align}
}
\end{lemma}

\begin{proof}
\textbf{Step (1)}. Denote $\mu=\min\limits_{f\in H-\{0\}}\mathscr{R}(f)=\frac{\int_0^d h(f')^2}{\int_{0}^d hf^2}$.  First assume $h\ge c>0$, choose $\varphi_k\in H$ such that $\mathscr{R}(\varphi_k)\to \mu, \mathscr{R}(\varphi_k)\le \mu+1$. Let $f_k=\frac{\varphi_k}{||\varphi_k||_{L^2(0,d)}}$. Then $||f_k||_{L^2}=1$, and 
\[\int_0^d (f_k')^2\le \frac{1}{c}\int_{0}^{d}h (f_k')^2\le \frac{\mu+1}{c}\int_{0}^{d}h (f_k)^2\le C\]

By Sobolev Compactness Theorem and and Banach-Alaoglu Theorem, there exists $G\in L^2, g\in L^2$, and $\{f_{n_k}\}\subset \{f_k\}$.
\begin{align}
		\lim_{k\to \infty}|f_{n_k}-G|_{L^2}= 0,\quad \lim_{k\to \infty}\int_0^d (f_{n_k}'-g)v=0, \quad \forall v\in L^2\nonumber
\end{align}
For any $v\in C_c^{\infty}$,
\begin{align}
		\int_0^d-Gv'=\int_0^d\lim_{k\to \infty}-f_{n_k}v'=\int_0^d\lim_{k\to \infty}f_{n_k}'v=\int_0^dgv.\nonumber
\end{align}
Thus $G'=g$ and $G\in H-\{0\}$.
\begin{align}
		\mu
		&=\lim_{k\to \infty}\frac{\int h(f_{n_k}')^2}{\int h(f_{n_k})^2}\geq \frac{\varliminf\limits_{k\rightarrow \infty} \int_0^d h(f_{n_k}')^2}{\int hG^2}\nonumber\\
		&\geq \frac{\varliminf\limits_{k\rightarrow \infty} \int_0^d h(G')^2+2hG'(f_{n_k}'-G')}{\int hG^2}=\frac{ \int_0^d h(G')^2}{\int hG^2}\nonumber
\end{align}
Let $0= \frac{d}{dt}\mathscr{R}(G+ tv)\big|_{t= 0}$, we get
\begin{align}
		\int_0^d G'\cdot v'hdt= \mu\int_0^d G\cdot vhdt ,  \quad \quad \forall v\in H-\{0\}.  \nonumber 
\end{align}
Note that the above equality holds for $v=c, c\in \R$. Assume $f\in W^{1,2}(0,d)$ and $f$ is not constant function, then 
\begin{align}
\int_0^d G'\cdot f'hdt=\int_0^d G'\cdot (f-\int_{0}^{d}hf)'hdt=\mu\int_0^d G\cdot (f-\int_{0}^{d}hf)hdt=\mu\int_0^d G\cdot fhdt.\nonumber
\end{align}

Thus
\begin{align}
	\int_0^d G'\cdot v'hdt= \mu\int_0^d G\cdot vhdt ,  \quad \quad \forall v\in W^{1,2}(0,d).  \nonumber 
\end{align}

\textbf{Step (2)}. If $h\ge 0$, we choose $h_k= (h^{\frac{1}{n- 1}}+1/k)^{n- 1}$. For $h_k$, from \textbf{Step (1)}, we can find $G_k$ such that 
\begin{align}
	\int_0^d G_k'\cdot v'h_kdt= \mu(k)\int_0^d G_k\cdot vh_kdt ,  \quad \quad \forall v\in W^{1,2}(0,d).  \nonumber 
\end{align}
By regularity theory, we have $G_k\in C^{\infty}[0,d]$. Choose $v\in C_c^{\infty}(0,d)$, we get $(h_kG_k')'+\mu(k)h_kG_k=0$. We assume $\sup_{x\in [0,d]}|G_k|=1, h'(\tau_0)=0$ and note $h_k^{\frac{1}{n- 1}}$ is concave, then
\begin{align}
	|G_k'(x)|=
	\begin{cases}
		&=|\frac{1}{h_k(x)}\int_{0}^x\mu(k)h_k(s)G_k(s)|\le Cx, \quad  x\in [0,\tau_0]\\
		&=|\frac{1}{h_k(x)}\int_{x}^d\mu(k)h_k(s)G_k(s)|\le C(d-x), \quad  x\in [\tau_0,d]
	\end{cases}
\end{align}
So $\sup|G_k|_{C^1[0,d]}<C$. 

Since $h_k^{\frac{1}{n- 1}}$ is concave, then for $x<\tau_0$, $h_k^{\frac{1}{n- 1}}(x)=\int_{0}^{x}(h_k^{\frac{1}{n- 1}})'\ge x(h_k^{\frac{1}{n- 1}})'(x)=xh_k'(x)h_k^{\frac{1}{n- 1}-1}(x)$. We get
\begin{align}
		|\frac{h_k'(x)}{h_k(x)}|=
		\begin{cases}
			&=|\frac{h_k'(x)}{\int_{0}^xh_k'(s)ds}|\le |\frac{h_k'(x)}{h_k(0)+\int_{0}^xh_k'(s)ds}|\le \frac{h_k'(x)}{xh_k'(x)}\le\frac{1}{x} , \quad  x\in (0,\tau_0]\\
			&\le \frac{1}{d-x}, \quad  x\in [\tau_0,d)
		\end{cases}
\end{align}

Then
\begin{align}
	|G_k''(x)|=|-\frac{h_k'(x)}{h_k(x)}G_k'(x)-\mu(k)G_k(x)|\le 
	|\frac{h_k'(x)}{h_k(x)}G_k'(x)|+C|G_k(x)|\le C.\nonumber
\end{align}
we have $|G_k|_{C^2[0,d]}<C$. By Arzela-Ascoli theorem we can choose $G_{n_k}\to \varphi$ in $C^1[0,d]$ norm for some $\varphi$. For any $f\in H-\{0\}$ we have
\begin{align}
	\frac{\int_0^d hf'^2}{\int_0^d hf^2}=\lim_{k\to \infty}\frac{\int_0^d h_kf'^2}{\int_0^d h_kf^2}\ge\varlimsup_{k\to \infty} \mu(k)\ge\lim_{k\to \infty}\frac{\int_0^d h_{n_k}G_{n_k}'^2}{\int_0^d h_{n_k}G_{n_k}^2}=\frac{\int_0^d h\varphi'^2}{\int_0^d h\varphi^2}\nonumber
\end{align}
Thus we have $\mu=\mathscr{R}(\varphi)$. Let $0= \frac{d}{dt}\mathscr{R}(\varphi+ tv)\big|_{t= 0}$, we get
\begin{align}
		\int_0^d \varphi'\cdot v'hdt= \mathscr{R}(\varphi)\int_0^d \varphi\cdot vhdt ,  \quad \quad \forall v\in H-\{0\}.\nonumber
\end{align}

Note that the above equality holds for $v=c, c\in \R$. Assume $f\in W^{1,2}(0,d)$ and $f$ is not constant function, then
\begin{align}
	\int_0^d \varphi'\cdot f'hdt=\int_0^d \varphi'\cdot (f-\int_{0}^{d}hf)'hdt=\mu\int_0^d \varphi\cdot (f-\int_{0}^{d}hf)hdt=\mu\int_0^d \varphi\cdot fhdt\nonumber.
\end{align}

Thus
\begin{align}
	\int_0^d \varphi'\cdot v'hdt= \mu\int_0^d \varphi\cdot vhdt ,  \quad \quad \forall v\in W^{1,2}(0,d).  \nonumber 
\end{align}
\end{proof}

\end{document}